\documentclass[12pt,reqno]{amsart}
\usepackage[utf8]{inputenc}
\usepackage[activeacute,english]{babel}

\usepackage[margin=0.9in]{geometry}
\usepackage[hidelinks]{hyperref} 

\usepackage[leqno]{amsmath}
\usepackage{enumerate}
\usepackage{amssymb}
\usepackage[alphabetic]{amsrefs}
\usepackage{mathtools}
\usepackage{physics}
\usepackage{amsthm}
\usepackage{mathrsfs}
\usepackage{verbatim}
\usepackage{physics}

\hypersetup{colorlinks=true, linkcolor=blue, citecolor=red}

\mathtoolsset{showonlyrefs}

\usepackage{pgfplots}
\pgfplotsset{compat=1.15}

\usetikzlibrary{arrows}

\theoremstyle{plain}
\newtheorem{theorem}{Theorem}[section]
\newtheorem{proposition}[theorem]{Proposition}
\newtheorem{corollary}[theorem]{Corollary}
\newtheorem{lemma}[theorem]{Lemma}
\theoremstyle{remark}

\theoremstyle{definition}
\newtheorem{definition}[theorem]{Definition}

\allowdisplaybreaks[1]

\numberwithin{equation}{section}

\makeatletter
\newcommand{\proofstep}[1]{%
  \par% ensure starting on a new paragraph
  \addvspace{\medskipamount}% some vertical space
  \textit{#1\@addpunct{.}}\enspace\ignorespaces
}
\makeatother

\newcommand{\smabs}[1]{| {#1} |}
\begin{document}

\title[Periodic Delaunay cylinders with constant ANMC]{Periodic Delaunay cylinders with constant\\anisotropic nonlocal mean curvature}

\author[F. Alcover]{Francesc Alcover}
\address{F. Alcover \textsuperscript{1}
\newline
\textsuperscript{1} 
DMI \& IAC3, University of the Balearic Islands, Cra. de Valldemossa, km. 7.5, E-07122 Palma, Illes Balears, Spain.}
\email{francesc.alcover@uib.cat}

\author[R. Bruera]{Renzo Bruera}
\address{R. Bruera \textsuperscript{1}
\newline
\textsuperscript{1}
Universitat Politècnica de Catalunya, Departament de Matem\`{a}tiques, Av. Diagonal 647, 08028 Barcelona.}
\email{renzo.bruera@upc.edu}

\date{\today}
\thanks{Francesc Alcover is supported by the MoMaLIP project PID2021-125711OB-I00 funded by MICIU/AEI/10.13039/501100011033 and the European Union NextGeneration EU/PRTR. This work is supported by the Spanish State Research Agency, through the Severo Ochoa and Mar\'{\i}a de Maeztu Program for Centers and Units of Excellence in R\&D (CEX2020-001084-M). Renzo Bruera is supported by the Spanish Ministry of Universities through the national program FPU (reference FPU20/07006), by the Spanish grants PID2021-123903NB-I00 and RED2022-134784-T funded by MCIN/AEI/10.13039/501100011033 and by ERDF “A way of making Europe”, and by the Catalan grant 2021-SGR-00087.}
\begin{abstract}
In this article we prove existence and symmetry properties of periodic surfaces of revolution with constant anisotropic nonlocal mean curvature, generalizing a classical result of Delaunay to the anisotropic nonlocal setting. 

First, by studying the corresponding periodic isoperimetric problem, under natural assumptions on the kernel, we use rearrangement inequalities to extend a periodic version of the Wulff inequality to the nonlocal setting. This leads to the existence and symmetry properties of minimizers for every given volume in each period, thus generalizing the results of \cite{CCM} to the anisotropic case.

Second, under the same hypotheses on the kernel, we prove the existence of a one-parameter family of Delaunay near-cylinders in $\mathbb{R}^2$ bifurcating from a straight cylinder and having each constant anisotropic mean curvature. This extends the results of \cite{CNLMCDelCyl} to the anisotropic case. The stability of these near-cylinders will be studied in a forthcoming paper.
\end{abstract}

\maketitle

%\tableofcontents

\section{Introduction}

For $n\geq 2$, let $K:\mathbb{R}^n \to [0,+\infty)$ be a nonnegative measurable function, and $E\subset \mathbb{R}^n$ be an open set. Whenever $K$ and $\partial E$ are sufficiently regular, we define the anisotropic nonlocal mean curvature (ANMC) of $E$ at $x\in \partial E$ as
\begin{equation}\label{ANMC_def}
    H_K[E](x) = -\mathrm{PV}\int_{\mathbb{R}^n} \left\{\chi_E(y)-\chi_{E^c}(y)\right\}K(x-y) \dd{y},
\end{equation}
where $E^c$ is the complement of $E$ in $\mathbb{R}^n$, $\chi_A$ denotes the characteristic function of a set $A$, and the integral is understood in the principal value sense. When $K(x)=\abs{x}^{-n-\alpha}$ for some $\alpha\in (0,1)$, the quantity \eqref{ANMC_def} is known simply as the nonlocal mean curvature (NMC). It arises when computing the first variation of the nonlocal perimeter, first introduced for the corresponding Dirichlet problem in the seminal paper by Cafferelli, Roquejoffre, and Savin \cite{CaffRoqSav}.

The ANMC weighs the interaction between different points not only by the distance between them but also by their relative orientation. Such interactions ---more precisely, its associated motion by mean curvature--- have been used to model dislocations (i.e., linear defects) in metallic crystals; see, for instance, \cite{AlvarezHochLeBouarMonneau2006}. It also arises when computing the first variation of the anisotropic nonlocal perimeter (see \cite{ChambolleMoriniPonsiglione2015}),
\begin{equation}\label{full_NL_perimeter}
    \mathrm{Per}_K[E] = \int_E \int_{E^c} K(x-y) \dd{x} \dd{y}.
\end{equation}

In this paper we address the existence and symmetry properties of \textit{anisotropic nonlocal Delaunay cylinders}---sometimes also called nonlocal unduloids---, i.e., surfaces of revolution that are $2\pi$-periodic in the $x_1$-direction and have constant ANMC, generalizing the classical result by Delaunay \cite{Delaunay}. In the first part of the paper, Section \ref{VASection}, we use symmetrization techniques to show, for every given volume $\omega>0$, the existence of minimizers of a periodic version of the anisotropic nonlocal perimeter among sets that are $2\pi$-periodic in the first coordinate, and have volume $\omega$ in each period. Moreover, we prove that all such minimizers have constant ANMC. This approach extends the results of Cabré, Csató, and Mas \cite{CCM} for $K(x)=\abs{x}^{-n-\alpha}$ to the anisotropic case. In the second part, Section \ref{sec:IFTSection}, we use the implicit function theorem to show the existence of a family of surfaces of revolution in $\mathbb{R}^2$ bifurcating from a straight cylinder\footnote{Of course, since we are now in the plane, these surfaces are just (pairs of) curves that are symmetric with respect to the $x_1$-axis.} that are periodic in the first coordinate and have constant ANMC. This construction generalizes that of Cabré, Fall, Solà-Morales, and Weth \cite{CNLMCDelCyl}, who dealt with the same problem for the homogeneous isotropic kernel $K(x)=\abs{x}^{-2-\alpha}$.

As already noted in \cite{CNLMCDelCyl}, the appropriate functional to study the (anisotropic) nonlocal perimeter of sets $E\subset \mathbb{R}^n$ that are $2\pi$-periodic in the $x_1$-direction is
\begin{equation}\label{def:PANP}
   \mathcal{P}_K[E]:=\int_{E\cap \Omega} \int_{E^c } K(x-y)\dd{x}\dd{y}.
\end{equation}
We call this functional the \textit{periodic anisotropic nonlocal perimeter} (PANP). Here, and throughout the whole paper, $\Omega = \{x\in \mathbb{R}^n : - \pi < x_1 < \pi \}$. Notice that $\mathcal{P}_K[E]$ does \textit{not} coincide with the more well-known \textit{anisotropic nonlocal perimeter relative to $\Omega$}, $\mathrm{Per}_K[E; \Omega]$, defined in \eqref{eq:per_K_decomposed}. In Subsection \ref{subsec::the_periodic_nonlocal_perimeter} we give a detailed justification of expression \eqref{def:PANP}.

The functional $\mathcal{P}_K$ was first introduced for the homogeneous isotropic kernel $K(x)=\abs{x}^{-n-\alpha}$ by Dávila, Del Pino, Dipierro, and Valdinoci in \cite{DavilaDelPinoDipierroValdinoci2016}, where nonlocal Delaunay cylinders appeared for the first time. Later, Cabré, Csató, and Mas \cite{CCM}, still for $K(x)=\abs{x}^{-n-\alpha}$, established that the periodic first variation of $\mathcal{P}_K$, and the Dirichlet first variation of $\mathrm{Per}_K[\ \cdot\ ;\Omega]$ both give rise to the NMC, introduced by Caffarelli, Roquejoffre, and Savin in \cite{CaffRoqSav}. In the Dirichlet case, it is well known (see \cite{CaffRoqSav}) that minimizers of $\mathrm{Per}_K[\ \cdot \ ; \Omega]$ have vanishing NMC, whereas in the periodic setting, since there is no exterior condition to be satisfied by a minimizer, we have $\inf \mathcal{P}_K=0$, where the infimum is taken over all open sets that are $2\pi$-periodic in $x_1$.\footnote{Indeed, consider, for example, in dimension $n$, the homogeneous isotropic kernel $K(x)=\abs{x}^{-n-\alpha}$ and the straight cylinders $\mathcal{C}_R = \{(x_1,x')\in \mathbb{R}\times \mathbb{R}^{n-1} : \abs{x'}<R\}$. Then, $\mathcal{P}_K[\mathcal{C}_{ R}]\to 0$ as $R \to 0$.} Therefore, it becomes necessary to introduce some constraint in the class of admissible sets. As in \cite{DavilaDelPinoDipierroValdinoci2016} and \cite{CCM}, we consider the isoperimetric problem for a given volume $\omega>0$. That is, given $\omega>0$, we only allow as competitors $2\pi$-periodic sets $F\subset \mathbb{R}^n$ such that
\begin{equation}
    \abs{F\cap \Omega} = \omega,
\end{equation}
where $\abs{\cdot}$ denotes the $n$-dimensional Lebesgue measure. As shown in \cite{CCM}, and also later on in this paper, this constraint leads to critical points of $\mathcal{P}_K$ (in particular, minimizers) having \textit{constant} ---but not necessarily vanishing--- ANMC. Conversely, every constant ANMC set that is $2\pi$-periodic in $x_1$ is a volume-constrained critical point of $\mathcal{P}_K$ ---this can be seen from \eqref{final_expression_in_euler_lagrange} for cylindrically symmetric variations, and by adapting the first variation formula in \cite[Theorem 6.1]{FigalliFuscoMaggiMillorMorini2015} for general volume-preserving variations; see \cite{CabreCsatoMas_stab} for more details.

\subsection{Assumptions on the anisotropic kernel}
Here we enumerate and motivate the assumptions on the kernel $K$ that we will need to prove our main results, Theorem \ref{main_theorem_var}, on the existence and symmetry properties of minimizers of the PANP, and Theorem \ref{mainresultperturbative}, on the existence of a bifurcated family of near-cylinders having constant ANMC, both stated below. Somewhat remarkably, since the techniques used to prove them are very different, the same key hypotheses on $K$, namely, \eqref{H1::Symetry}---\eqref{H5::UpperBound}, lead to our proofs of both theorems ---still, as we mention later on, our proof of Theorem \ref{mainresultperturbative} requires some additional technical assumptions on $K$; namely, \eqref{H6::DecayLipschitzNorm}---\eqref{H7::CloseToHom_2}.

At the end of this subsection we describe the important class of anisotropic kernels defined by a monotonic norm. These kernels satisfy all hypotheses \eqref{H1::Symetry}---\eqref{H7::CloseToHom_2}. As we explain below, the anisotropic nonlocal perimeter associated to a kernel defined by a norm is a natural nonlocal generalization of the local notion of anisotropic perimeter.

First, we assume $K$ to be symmetric,
\begin{equation}\label{H1::Symetry}
    K(x) = K(-x) \text{ for all } x \in \mathbb{R}^n.
\end{equation}
This means that the interaction of, say, point $A$ with point $C$ in Figure \ref{fig1} below must be the same as the interaction of point $C$ with point $A$.

The following two conditions on $K$ arise from the fact that we are looking for cylindrically symmetric surfaces with constant ANMC. The first one is rather natural: we will assume that $K$ is invariant under rotations about the $x_1$-axis. That is, with a slight abuse of notation, and denoting $x=(x_1,x') \in \mathbb{R} \times \mathbb{R}^{n-1}$, we assume that
\begin{equation}\label{H2::RotInv}
    K(x_1,Rx')=K(x_1,x') \text{ for all } R \in \operatorname{O}(n-1),
\end{equation}
where $\operatorname{O}(n-1)$ denotes the group of orthogonal matrices of order $n-1$. In other words, $K$ depends only on $x_1$ and $|x'|$.

The third assumption, \eqref{H3::RadiallyDec} below, turns out to be essential when proving that the minimizers of $\mathcal{P}_K$ must be cylindrically symmetric, and also in the perturbative argument in Section~\ref{sec:IFTSection}. We will assume that the map
\begin{equation}\label{H3::RadiallyDec}
    \mathbb{R}^{n-1}\ni x' \mapsto K(x_1,x') \text{ is radially decreasing in } \abs{x'}\in (0,+\infty) \text{ for every } x_1 \in \mathbb{R}.
\end{equation}
In other words, denoting $K(x_1,x') = \tilde{K}(x_1,\abs{x'})$ for some $\tilde{K}:\mathbb{R}^2 \to [0,\infty)$, we have that $\tilde{K}(z_1,\cdot)$ is decreasing in $(0,+\infty)$ for every $z_1\in \mathbb{R}$. As we will explain in the next subsection, both \eqref{H2::RotInv} and \eqref{H3::RadiallyDec} play a key role in the proof that $\mathcal{P}_K$ does not increase under cylindrical symmetric decreasing rearrangements about the $x_1$-axis.

The following hypotheses are more technical. To show the existence of minimizers of $\mathcal{P}_K$, we assume that
\begin{equation}\label{H4::int}
    K(x_1,\cdot) \in L^1(\mathbb{R}^{n-1}) \text{ for all } x_1 \in \mathbb{R} \backslash \{ 0 \}.
\end{equation}
Of course, \eqref{H4::int} follows from the stronger condition \eqref{H5::UpperBound} below. However, as we will see, to prove the existence of minimizers it is enough to assume \eqref{H4::int}.

We also assume that $K$ is comparable to the homogeneous isotropic kernel with parameter $\alpha \in (0,1)$, i.e., that
\begin{equation}
\lambda |x|^{-n-\alpha} \le K(x) \label{H5::LowerBound} \end{equation}
and
\begin{equation}
K(x) \le \Lambda |x|^{-n-\alpha} \label{H5::UpperBound}
\end{equation}
for every $x\in \mathbb{R}^n$, for some $\lambda,\Lambda>0$ and some $\alpha \in (0,1)$. The second condition ensures that the ANMC is well defined whenever $\partial E$ is sufficiently regular.

The previous hypotheses are needed both in Section \ref{VASection} and Section \ref{sec:IFTSection}. However, in the latter we need to introduce some additional assumptions on $K$, which we state next. To make it clear that we are now in dimension $n=2$, slightly abusing notation we will write  $K(z)=K(z_1,z_2)$ for $z\in \mathbb{R}^2$. 

In order to have the required smoothness to apply the implicit function theorem, we need to have appropriate decay of the $C^{1,1}$ norm of $K(t,\cdot)$. Hence, we assume that for every $t>0$,
\begin{equation}\label{H6::DecayLipschitzNorm}
    K(t,\cdot)\in C^{1,1}(\mathbb{R}) \quad \text{and} \quad
    t\norm{K_{z_2}(t,\cdot)}_{L^\infty(\mathbb{R})}+t^2[K_{z_2}(t,\cdot)]_{C^{0,1}(\mathbb{R})}\leq \Lambda t^{-2-\alpha}.
\end{equation}

In Section \ref{sec:IFTSection}, to apply the implicit function theorem it will be important to show that the linearization of the ANMC operator at a straight cylinder is nondegenerate. In order to do this, we further require that $K$ be close to homogeneous in either of the following two senses. A first sufficient assumption is that there exists a nonnegative $-(2+\alpha)$-homogeneous kernel $K_0$ satisfying \eqref{H3::RadiallyDec} and \eqref{H5::UpperBound} (for instance, $K_0(z)=\abs{z}^{-2-\alpha}$), such that
\begin{equation}\label{H7::CloseToHom_1}
    \abs{K(z)-K_0(z)} \leq \varepsilon K_0(z)
\end{equation}
for every $z\in \mathbb{R}^2$, for some $\varepsilon >0$ small enough. A second sufficient requirement is that $K$ be differentiable away from the origin and that
\begin{equation}\label{H7::CloseToHom_2}
\abs{z \cdot \nabla  K(z) + (2+\alpha)K(z)} \leq \varepsilon K(z)
\end{equation}
for every $z\in \mathbb{R}^2$, for some $\varepsilon>0$ small enough.\footnote{Inequality \eqref{H7::CloseToHom_1} can be deduced from \eqref{H7::CloseToHom_2} with $K_0(z)=K(z/\abs{z})\abs{z}^{-2-\alpha}$ (for a possibly different $\varepsilon$). However, $K_0$ may fail to satisfy \eqref{H3::RadiallyDec}.}

Proving nondegeneracy of the linearized operator is a delicate problem. As it is well known, and we will show in Section \ref{sec:IFTSection}, in dimension 2 the linearization of the ANMC operator on a straight cylinder is an integro-differential operator of order $1+\alpha$ that acts on periodic functions of a single variable. When $K(z)=\abs{z}^{-2-\alpha}$, this operator is precisely the one-dimensional fractional Laplacian of order $1+\alpha$, plus some 0-order terms. The nondegeneracy of this operator can be established by direct calculation. We will be able to do it also for kernels satisfying \eqref{H7::CloseToHom_1} or \eqref{H7::CloseToHom_2}, despite the fact that then the linearized operator then is not the fractional Laplacian. However, for general kernels, the question remains unanswered:
\begin{itemize}
\item[]
{\bf Open problem.} For which class of kernels is the linearization of the ANMC \newline
operator on a straight cylinder nondegenerate?
\end{itemize}
A related nontrivial result on nondegeneracy is that of Frank and Lenzmann \cite{FrankLenzmann_2013}: it concerns operators of the form $(-\Delta)^s - V(\abs{x})$ in $\mathbb{R}$. The open problem above would have as an analogue the extension of the results of \cite{FrankLenzmann_2013} to more general integro-differential operators with a 0 order term.

We note that assumptions \eqref{H6::DecayLipschitzNorm} and \eqref{H7::CloseToHom_2} are similar to those required by Cabré, Mas, and Solà-Morales in \cite{CabreMasSolaMorales_AC_BenjOno}, where, using the same techniques as in Section \ref{sec:IFTSection} in the present paper, they prove the existence of periodic solutions to some one-dimensional semilinear integro-differential equations. Similarly to our assuming \eqref{H6::DecayLipschitzNorm}, the authors require their kernel $\tilde{K}$---which is defined on the real line---to satisfy
\begin{equation}
    |t^k\partial^{(k)}_t(t^{1+\alpha}\tilde{K}(t))|\leq C
\end{equation}
for every $t>0$ and $k=1,2,3,4$, to ensure that a certain operator is sufficiently smooth. They also assume that
\begin{equation}\label{tKt_dec}
    \partial_t(t \tilde{K}(t))< 0
\end{equation}
for every $t>0$. Like us ---we will explain this more in detail in Section \ref{sec:IFTSection}---, the authors need to show that the Fourier multipliers of some integral operator defined by the kernel are decreasing with respect to their order or index. This is ensured by \eqref{tKt_dec} in their case, and by \eqref{H7::CloseToHom_2} in ours. Notice that \eqref{H7::CloseToHom_2} can be rewritten as $(-\alpha - \varepsilon)K(z) \leq \nabla \cdot (zK(z)) \leq (-\alpha + \varepsilon)K(z)<0$. This shows that our assumption \eqref{H7::CloseToHom_2} is stronger than \eqref{tKt_dec}. This is because our linearized operator involves an additional 0 order term due to the interactions with the points of the cylinder lying in the lower half plane, which does not appear in \cite{CabreMasSolaMorales_AC_BenjOno}.

\subsubsection{Kernels defined by a monotonic norm.}\label{subsubsec:kernels_defined_by_an_monotonic_norm} An important class of anisotropic kernels satisfying all hypotheses \eqref{H1::Symetry}---\eqref{H7::CloseToHom_2} is the one of kernels defined by a monotonic norm:
\begin{equation}
    K(x)=\norm{(x_1,\abs{x'})}_K^{-n-\alpha},
\end{equation}
where $\norm{\cdot}_K$ is a monotonic norm in $\mathbb{R}^2$. We say that a norm $\norm{\cdot}_K$ is \textit{monotonic} if, given $x,y\in \mathbb{R}^n$ such that $\abs{x_i}\leq \abs{y_i}$ for every $i=1,\dots,n$, it holds that $\norm{x}_K\leq \norm{y}_K$.\footnote{In the literature, these norms are often called \textit{absolutely monotonic}. However, to avoid confusion with the notion of \textit{completely monotonic functions} appearing in Theorem \ref{theorem_monotonicity_symmetric_rearrangement} we simply call them monotonic.} As shown in \cite{BauerStoer1961}, this condition is in fact equivalent to having $\norm{(x_1,\dots,x_n)}_K=\norm{(\abs{x_1},\dots,\abs{x_n})}_K$ for every $x\in \mathbb{R}^n$. Geometrically, this means that the unit ball $\{ x\in \mathbb{R}^2 : \norm{x}_K < 1 \}$ is symmetric with respect to both coordinate axes.

Thanks to the equivalent characterizations discussed above, it is immediate to see that the hypotheses \eqref{H1::Symetry}---\eqref{H3::RadiallyDec} are all satisfied by every such kernel. Clearly, \eqref{H5::LowerBound} and \eqref{H5::UpperBound} also hold because all norms on $\mathbb{R}^2$ are equivalent, and \eqref{H6::DecayLipschitzNorm}, \eqref{H7::CloseToHom_1}, and \eqref{H7::CloseToHom_2} are all satisfied with $\varepsilon = 0$ because $K$ is homogeneous (provided that it is sufficiently smooth).

Anisotropic kernels that are defined by norms are important because their associated nonlocal perimeter coincides, after multiplication by a normalizing factor, with the classical (i.e., local) anisotropic perimeter in the limit $\alpha\uparrow 1$; see \cite{Ludwig2014}. In the local case, the solution to the corresponding isoperimetric problem is a convex set known as the \textit{Wulff shape} associated to the norm $\norm{\cdot}_K$, and it coincides with the unit ball of its dual norm. Our first main result, Theorem \ref{main_theorem_var} below, shows the existence of a nonlocal version of the Wulff shape for sets that are periodic in one variable.

\subsection{Periodic minimizers with constant anisotropic nonlocal mean curvature via a variational approach}

In Section \ref{VASection} we establish the existence of periodic constant ANMC surfaces in $\mathbb{R}^n$. By a variational method we show the existence of a minimizer of the periodic anisotropic perimeter,
\begin{equation} \label{eq:PANP_def}
   \mathcal{P}_K[E]:=\int_{E\cap \Omega} \int_{E^c } K(x-y)\dd{x}\dd{y},
\end{equation}
among sets $E \subset \mathbb{R}^n$ which are periodic in the coordinate $x_1$ and have a prescribed volume within the slab $\Omega=\{ x\in\mathbb{R}^n : -\pi < x_1 < \pi \}$. Moreover, we prove that volume-constrained minimizers of \eqref{eq:PANP_def} are cylindrically symmetric and have constant ANMC whenever they are sufficiently regular.

The main result of Section \ref{VASection}, which is a generalization of \cite[Theorem 1.1]{CCM}, is the following. Here, and throughout the whole paper, we use the notation $x=(x_1,x')\in \mathbb{R}\times \mathbb{R}^{n-1}$.
\begin{theorem}\label{main_theorem_var}
Let $n\geq 2$ and $K:\mathbb{R}^n \to (0,+\infty)$ be a positive measurable kernel satisfying \eqref{H1::Symetry}---\eqref{H5::LowerBound}, and let $\mathcal{P}_K$ be the associated periodic anisotropic nonlocal perimeter, as defined in \eqref{eq:PANP_def}.

For every $\omega>0$, there exists a minimizer $E\subset \mathbb{R}^n$ of $\mathcal{P}_K$ among all measurable sets $F\subset \mathbb{R}^n$ which are $2\pi$-periodic in the first coordinate and satisfy
\begin{equation}
    \abs{F\cap \Omega}=\omega,
\end{equation}
with $\Omega = \{x\in \mathbb{R}^n: -\pi < x_1 < \pi \}$. Moreover, up to sets of measure zero:
\begin{enumerate}[(i)]
    \item Up to a translation in the $x'$-coordinate, every minimizer $E$ is of the form
    \begin{equation}
        E=\{(x_1,x')\in \mathbb{R}^{n} : \abs{x'}<u(x_1)\}
    \end{equation}
    for some $2\pi$-periodic nonnegative function $u$ whose power $u^{n-1}$ belongs to $W^{\alpha,1}(-\pi,\pi)$. In particular, $u\in W^{\frac{\alpha}{n-1},n-1}(-\pi,\pi)$. 
\end{enumerate}
Assume that $K$ satisfies, in addition, the upper bound \eqref{H5::UpperBound}. Then:
\begin{enumerate}[(i)]
\setcounter{enumi}{1}
    \item If $K$ satisfies at least one of the hypotheses \textit{(1)'} or \textit{(2)'} in the second part of Theorem \ref{theorem_monotonicity_symmetric_rearrangement}, then, up to a translation in the $x_1$-direction, $u$ is even with respect to $0$, as well as nonincreasing in $(0,\pi)$.
    \item For every minimizer $E$ there exists a constant $\theta\in \mathbb{R}$ such that, for every $x\in \partial E$, if $u$ is of class $C^{1,\beta}$ in a neighborhood of $x_1$ for some $\beta> \alpha$, and $u(x_1)>0$, then $H_K[E](x)=\theta$.
    \item For $\omega$ small enough, depending only on $n$, $\alpha$, $\lambda$, and $\Lambda$, the function $u$ is nonconstant. As a consequence, for $\omega$ small enough, straight cylinders are not minimizers.
\end{enumerate}
\end{theorem}

We now make a few comments on the result above. We start by claim $(iii)$. Recall that, by \eqref{H5::LowerBound} and \eqref{H5::UpperBound}, the ANMC is a nonlocal operator of order $1+\alpha$ on $\partial E$, and therefore it is well defined at points $x\in \partial E$ around which $\partial E$ is of class $C^{1,\beta}$ for some $\beta>\alpha$. We do \emph{not} claim any regularity of minimizers ---in fact, $C^{1,\beta}$ regularity theory for constrained minimizers has not been completed yet, even in the case $K(x)=\abs{x}^{-n-\alpha}$. 

Next, following \cite{CCM}, the proof of $(iv)$ is based on a simple comparison argument: we show that, for small enough volumes, a $2\pi$-periodic array of balls, each one of volume $\omega$, has smaller PANP than the straight cylinder $\mathcal{C}_\omega=\{(x_1,x')\in \mathbb{R}^n : \abs{x'}<R_{\omega}\}$, with $R_{\omega}$ such that $\mathcal{C}_{\omega}$ has volume $\omega$ in $\Omega$.

Let us now turn to the symmetry properties in $(i)$ and $(ii)$. Claim $(i)$ states that minimizers are cylindrically symmetric about the $x_1$-axis, and claim $(ii)$ states that, under additional hypotheses on $K$, minimizers are symmetric with respect to the hyperplane $\{x_1=0\}$ ---more precisely, that the generatrix function $u$ is even with respect to $0$, and nonincreasing in $(0,\pi)$. The proofs of these two results are based on two isoperimetric inequalities that we state next. To do so, we need to introduce the concepts of \emph{cylindrical} and \emph{symmetric decreasing periodic rearrangements}, which will be described more in detail in Subsection~\ref{SubsectionNotation}.

The cylindrical rearrangement of a set $E\subset \mathbb{R}^n$ is the unique set $E^{*\mathrm{cyl}}\subset \mathbb{R}^n$ such that its transversal sections $(E^{*\mathrm{cyl}})_{x_1}:=\{ x'\in \mathbb{R}^{n-1} : (x_1,x')\in E^{*\mathrm{cyl}} \}$ are open balls in $\mathbb{R}^{n-1}$ centered at $x'=0$ and with the same $(n-1)$-dimensional Lebesgue measure as $E_{x_1}$. The first symmetry result, Lemma \ref{lemma_cylindrical} below, states that cylindrical rearrangements of a set do not increase its periodic anisotropic nonlocal perimeter. It was first proved by Cabré, Csató, and Mas \cite{CCM} for the case $K(x)=\abs{x}^{-n-\alpha}$. Notice that this result does not require the set $E$ to be periodic.
\begin{lemma}\label{lemma_cylindrical}
Let $E$ be a measurable set such that $\abs{E\cap \Omega  }<+\infty$, and let $K:\mathbb{R}^n \to (0,+\infty)$ be a positive measurable kernel satisfying \eqref{H2::RotInv}, \eqref{H3::RadiallyDec}, and  \eqref{H4::int}.

Then, $\mathcal{P}_K[E^{*\mathrm{cyl}}]\leq \mathcal{P}_K[E]$. 

Moreover, if $\abs{E\cap \Omega}>0$ and $\mathcal{P}_K[E]<+\infty$, equality $\mathcal{P}_K[E^{*\mathrm{cyl}}]= \mathcal{P}_K[E]$ holds if, and only if, $E=E^{*\mathrm{cyl}}+(0,c)$ for some $c\in \mathbb{R}^{n-1}$ up to a set of measure 0.
\end{lemma}
The proof of the inequality is a relatively simple application of Riesz's rearrangement inequality to every transversal section $E_{x_1}$, and is based on the observation that 
\begin{equation}
    \mathcal{P}_K[E] = \int_{-\pi}^\pi \dd{x_1}\int_{\mathbb{R}}\dd{y_1} \left\{\int_{E_{x_1}}\dd{x'}\int_{\mathbb{R}^{n-1}}\dd{y'} - \int_{E_{x_1}}\dd{x'}\int_{E_{y_1}}\dd{y'}\right\} K(x-y).
\end{equation}
The characterization of the case of equality is more delicate, and uses the strict version of Riesz's rearrangement inequality. Claim $(i)$ in Theorem \ref{main_theorem_var} is a direct consequence of this result.

On the other hand, given a set $E\subset \mathbb{R}^n$ that is $2\pi$-periodic in the $x_1$ coordinate, its symmetric decreasing periodic rearrangement is the unique set $E^{*\mathrm{per}}\subset \mathbb{R}^n$ that is $2\pi$-periodic in the $x_1$-direction and such that its longitudinal sections $E^{*\mathrm{per}} \cap \{x'=c\}\cap \Omega$, $c\in \mathbb{R}^{n-1}$, are line segments centered at $x_1=0$ with the same 1-dimensional Lebesgue measure as $E \cap \{x'=c\}\cap \Omega$. Our second main result, Theorem \ref{theorem_monotonicity_symmetric_rearrangement} below, states that $\mathcal{P}_K$ does not increase by periodic symmetric decreasing rearrangements, and it characterizes the case of equality. It is a generalization of \cite[Theorem 1.2]{CCM}, which concerns the case $K(x) = \abs{x}^{-n-\alpha}$.

\begin{theorem}\label{theorem_monotonicity_symmetric_rearrangement}
    Let $K:\mathbb{R}^n\to [0,+\infty)$ be a nonnegative measurable kernel satisfying the upper bound \eqref{H5::UpperBound}. Let $E\subset \mathbb{R}^n$ be a measurable set that is $2\pi$-periodic in the first coordinate, $x_1$, and such that $\mathcal{P}_K[E]<+\infty$. Suppose, in addition, that $K$ satisfies, for almost every $x'\in \mathbb{R}^{n-1}$, one of the following:
\begin{enumerate}[(1)]
    \item $K(\cdot,x')$ is convex in $(0,+\infty)$,
    \item The function $\tau>0\mapsto K(\sqrt{\tau},x')$ is completely monotonic,\footnote{A function $f:[0,+\infty)\to [0,+\infty)$ is called \textit{completely monotonic} if it has derivatives of all orders and $(-1)^n \derivative[n]{f(t)}{t} \geq 0$ for every $n\geq 1$ and every $t>0$.}
    \item $K(\cdot,x')$ is nonincreasing in $(0,\pi)$ and vanishes in $[\pi,+\infty)$.
\end{enumerate}

    Then, $\mathcal{P}_K[E^{*\mathrm{per}}] \le \mathcal{P}_K[E]$. 
    
    If, moreover, $K$ satisfies, for almost every $x'\in \mathbb{R}^{n-1}$, one of the following:
    \begin{enumerate}[(1)']
    \item $K(\cdot,x')$ is convex in $(0,+\infty)$ and strictly convex\footnote{Given an interval $I\subset \mathbb{R}$, we say that a function $f:I\to \mathbb{R}$ is strictly convex in $I$ if, for every $t_1\neq t_2$ in $I$ and every $\lambda\in (0,1)$, $f(\lambda t_1 + (1-\lambda)t_2)<\lambda f(t_1)+(1-\lambda)f(t_2)$.} in $(c,c+2\pi)$ for some $c\geq 0$,
    \item $K$ is as in (2) and nonzero,
    \item $K(\cdot,x')$ is decreasing in $(0,\pi)$ and vanishes in $[\pi,+\infty)$,
    \end{enumerate}
    then, equality $\mathcal{P}_K[E^{*\mathrm{per}}] = \mathcal{P}_K[E]$ holds if, and only if, $E= E^{* \mathrm{per}} + c e_1$ for some $c \in \mathbb{R}$, with $e_1 := (1,0,...,0)\in \mathbb{R}^n$, up to a set of measure zero.
\end{theorem}

Note that the homogeneous and isotropic kernel $K(x)=\abs{x}^{-n - \alpha}$ satisfies \textit{(2)}, but not \textit{(1)} (unless $n=1$) or \textit{(3)}. The case of completely monotonic kernels, that is, those satisfying \textit{(2)}, was introduced by Cabré, Csató, and Mas in the paper \cite{CCM} on the isotropic nonlocal perimeter. Instead, \textit{(1)} and \textit{(3)}, as well as \textit{(1)'} and \textit{(3)'}, appeared in \cite[Theorem 1.4]{PerIntDiffEq}, on the study of periodic solutions to \textit{semilinear} integro-differential equations. Unlike in the nonlocal perimeter setting, where the functional is given by \eqref{def:PANP}, the functional in \cite{PerIntDiffEq} has the generalized Gagliardo seminorm $\int\int (u(x)-u(y))^2K(x-y)\dd{x}\dd{y}$ as functional under study.

In the same paper, the authors show that the classes of kernels satisfying, respectively, conditions \textit{(1)}, \textit{(2)}, and \textit{(3)}, do not contain and are not contained in each other. The most general ``natural'' class of kernels for which Theorem~\ref{theorem_monotonicity_symmetric_rearrangement} may hold remains an open question. 

We remark that there are important examples of anisotropic kernels satisfying all assumptions \eqref{H1::Symetry}---\eqref{H7::CloseToHom_2} for which none of the conditions \textit{(1)}, \textit{(2)} or \textit{(3)} hold ---for instance, the family of homogeneous kernels $K(x) = \norm{(x_1, \abs{x'})}_{\ell^p}^{-n-\alpha}$, with $\norm{(z_1,z_2)}_{\ell^p } = (\abs{z_1}^p+\abs{z_2}^p)^{1/p}$, $p> 1$. Interestingly, however, the case $p=1$ does satisfy \textit{(1)}, in contrast to the isotropic kernel $K(x)=\abs{x}^{-n - \alpha}$.

The proof of Theorem \ref{theorem_monotonicity_symmetric_rearrangement} relies on a Riesz rearrangement inequality on the circle ---Theorem \ref{TheoremRearrangementIneqCircle} below--- applied to every longitudinal section $E_{x'}:=\{x_1 \in \mathbb{R} : (x_1,x')\in E\}$, and, similarly to Lemma \ref{lemma_cylindrical}, is based on the observation that, if $E$ is periodic, then
\begin{align}
    \mathcal{P}_{K}[E] 
    &=\int_{\mathbb{R}^{n-1}} \dd{x'} \int_{\mathbb{R}^{n-1}} \dd{y'} \left\{\int_{E_{x'}\cap (-\pi,\pi)} \dd{x_1}  \int_{\mathbb{R}} \dd{y_1} - \int_{E_{x'}\cap (-\pi,\pi)} \dd{x_1}  \int_{ E_{y'}} \dd{y_1} \right\} K(x-y)\\
    &=\int_{\mathbb{R}^{n-1}} \dd{x'} \int_{\mathbb{R}^{n-1}} \dd{y'} \begin{multlined}[t]
        \left\{\int_{E_{x'}\cap (-\pi,\pi)} \right. \dd{x_1} \int_{\mathbb{R}} \dd{y_1} K(x-y)\\ \left. - \sum_{k\in \mathbb{Z}}\int_{E_{x'}\cap (-\pi,\pi)} \dd{x_1}  \int_{ E_{y'}\cap (-\pi, \pi )} \dd{y_1} K(x-y-2\pi ke_1)\right\}.
    \end{multlined}  
\end{align}
Conditions \textit{(1)--(3)} arise from the fact that, in order to apply the periodic rearrangement inequality, Theorem \ref{TheoremRearrangementIneqCircle} below, we need that the function
\begin{equation}
    g_{x',y'}(s) = \sum_{k\in\mathbb{Z}}K(s+2\pi k,x'-y')
\end{equation}
be nonincreasing in $(0,\pi)$ for almost every $(x',y')\in \mathbb{R}^{n-1}\times \mathbb{R}^{n-1}$. The case of equality can be characterized if $g$ is decreasing in $(0,\pi)$.\footnote{In order to avoid ambiguity, we point out that by \textit{decreasing} we always mean \textit{strictly decreasing}.} Similarly to \cite[Theorem 1.4]{PerIntDiffEq}, in Lemma \ref{lemma_monotonicity_g} we show that $g$ is indeed nonincreasing (resp. decreasing) in $(0,\pi)$ whenever $K$ satisfies at least one of the assumptions \textit{(1)}, \textit{(2)}, or \textit{(3)} (resp. \textit{(1)'}, \textit{(2)'}, or \textit{(3)'}) in Theorem \ref{theorem_monotonicity_symmetric_rearrangement}.

Class \textit{(3)} is not considered in the statement of Theorem \ref{main_theorem_var} since its assumption \eqref{H5::LowerBound} is incompatible with condition \textit{(3)} in Theorem \ref{theorem_monotonicity_symmetric_rearrangement}. 

The proofs of Theorems \ref{main_theorem_var} and \ref{theorem_monotonicity_symmetric_rearrangement} follow closely those of \cite[Theorem 1.1]{CCM} and \cite[Theorem 1.2]{CCM}. We include mostly full proofs of our results for completeness.

Finally, we mention that the functional counterparts of the strict isoperimetric inequalities contained in Lemma \ref{lemma_cylindrical} and Theorem \ref{theorem_monotonicity_symmetric_rearrangement} on the nonlocal perimeter, that is, the corresponding (strict) Pólya-Szeg\H{o} type inequalities for the periodic Gagliardo seminorm, were recently proved by Csató and Mas \cite{CsatoMas2025} in the homogeneous isotropic case. One would expect that their methods extend to more general ``anisotropic'' seminorms under similar assumptions on the kernel as those considered in the present paper.

\subsection{Periodic near-cylinders with constant anisotropic nonlocal mean curvature in $\mathbb{R}^2$ through the implicit function theorem}
%\subsection{Periodic constant ANMC curves on the plane through the implicit function theorem}

In Section \ref{sec:IFTSection} we establish, by means of a perturbative argument, the existence of a continuous family of periodic Delaunay sets in the plane, each having constant ANMC, bifurcating from a straight cylinder. We point out that the existence of these curves is a purely nonlocal phenomenon, since in the local setting the only curves having constant curvature are straight lines and circular arcs. Our result extends that of Cabré, Fall, Solà-Morales, and Weth \cite{CNLMCDelCyl}, in which they treat the case of the homogeneous isotropic kernel $K(z)=\abs{z}^{-2-\alpha}$. 

Since in the previous section we showed that volume-constrained minimizers of $\mathcal{P}_K$, under appropriate assumptions on the kernel, must be cylindrically symmetric and have constant ANMC, we now look for sets of the form
\begin{equation} \label{bandsexpression}
    E = \{ (z_1,z_2) \in \mathbb{R}^2 \; : \; -u(z_1) < z_2 < u(z_1) \},
\end{equation}
where $u: \mathbb{R} \to (0,+\infty)$ is a positive function such that $\partial E$ has constant ANMC. We achieve this by perturbing a straight band, $\{ (z_1,z_2)\in \mathbb{R}^2: -R < z_2 < R \}$, ---which we will see has positive constant ANMC--- for some appropriately chosen $R>0$, in such a way that we obtain a continuous family of sets such that every one of them is periodic in the first coordinate, $z_1$, and has constant ANMC. 

Throughout the whole text, we denote by $C^{1,\beta}_{even}(\mathbb{T}^1)$ the space of $2\pi$-periodic, even functions of class $C^{1,\beta}$ endowed with the usual norm. That is,
\begin{equation}
    C^{1,\beta}_{even}(\mathbb{T}^1) := \{ u \in C^{1,\beta}(\mathbb{R}) : \ u \text{ is }2\pi\text{-periodic and even} \},
\end{equation}
with $\norm{u}_{C^{1,\beta}_{even}(\mathbb{T}^1)} := \norm{u}_{C^{1,\beta}([-\pi,\pi])}$. The space $L^{2}_{even}(\mathbb{T}^1)$ is defined similarly. 

Our last main result is the following.
\begin{theorem} \label{mainresultperturbative}
    Let $K:\mathbb{R}^2 \to (0,+\infty)$ be a positive measurable kernel satisfying \eqref{H1::Symetry}---\eqref{H6::DecayLipschitzNorm}. Suppose that $K$ satisfies, in addition, either \eqref{H7::CloseToHom_1} or \eqref{H7::CloseToHom_2} for some $\varepsilon>0$ small enough (as given by Proposition \ref{linearizedOpInvCont}).
    
    Then, there exist $R^*>0$ and $\nu>0$, both depending only on $K$, and a continuous family parametrized by $a\in(-\nu,\nu)$ of periodic and even functions, $w(a): \mathbb{R} \to \mathbb{R}$, each with prime period equal to $2\pi$, for which the sets
    \begin{equation}
       E(a) = \{ (z_1,z_2) \in \mathbb{R}^2 : -w(a)(z_1) < z_2 < w(a)(z_1) \}
    \end{equation}
    have each constant anisotropic nonlocal mean curvature equal to $h_{\gamma(a)R^*}$, where $h_{\rho}$ denotes the anisotropic nonlocal mean curvature of the straight cylinder of radius $\rho>0$, and $\gamma(a)R^*$, appearing below, is the leading order of the 0$^{\text{th}}$ Fourier coefficient of $w(a)$ as $a\to 0$.
    
    In addition, every $w(a)$ is of class $C^{1,\beta}$ for some $\beta \in (\alpha,1)$, and is of the form
    \begin{equation}
        w(a)(s) = \gamma(a) R^* + a (\cos(s)+v(a)(s)),
    \end{equation}
    with $v(a)$ orthogonal to $\cos(\cdot)$ in $L^2_{even}(\mathbb{T}^{1})$, and $(\gamma(a),v(a)) \xrightarrow[a\to 0]{} (1,0)$ in $\mathbb{R}\times C^{1,\beta}_{even}(\mathbb{T}^1)$.
\end{theorem}

\begin{comment}D'acord, ho coment.
In contrast to \cite{CNLMCDelCyl}, in our setting we can not assure the stronger property of the surfaces of the family having the same ANMC. However, the property does hold if the kernel is homogeneous\footnote{The anhomogeneous isotropic case, with the stronger property, was also proven in F. Alcover's Master's Thesis \cite{TFMFrancesc}.}, as the next result states. \textcolor{blue}{Aquest paràgraf no m'agrada: sí, a [4] i al teu TFM tots els $E(a)$ tenen la mateixa constant ANMC, però al preu de canviar el període. Aquí tots els $E(a)$ tenen el mateix període ---en correspondència amb la primera part. "It's not a bug, it's a feature". És a dir, que a [4] i el teu TFM tots tinguin la mateixa constant ANMC és gràcies a que si $K$ és homogeni la $H_K$ reescala bé, però no ho veig una propietat "intrínseca", sinó més una casualitat.}
\end{comment}

In the important case that $K$ is homogeneous (for example, the class of kernels described in Subsubsection \ref{subsubsec:kernels_defined_by_an_monotonic_norm}), the conclusion of Theorem \ref{mainresultperturbative} can be strengthened as follows.
\begin{corollary}\label{corollary_hom_K}
    In the setting of Theorem \ref{mainresultperturbative}, if $K$ is homogeneous, then the $w(a)$ can be rescaled so that the sets $E(a)$ have all the same nonlocal mean curvature equal to $h_R$. 

    Explicitly, if $K$ is homogeneous, then the sets
    \begin{equation}
    \tilde{E}(a):=\gamma(a)E(a) = \{(z_1,z_2) : -\tilde{w}(a)(z_1)<z_2<\tilde{w}(a)(z_1)\},
    \end{equation}
    with $\tilde{w}(a)(\cdot):=-\gamma(a)^{-1}w(a)(\gamma(a)\cdot )$, have all constant nonlocal mean curvature equal to $h_{R^*}$. Additionally, $\tilde{E}(a)\neq\tilde{E}(a')$ for $a\neq a'$.
\end{corollary}

The proof of Theorem \ref{mainresultperturbative} is based on a Lyapunov-Schmidt procedure applied to the ANMC operator. The first step is to derive a more treatable expression for the ANMC operator acting on sets of the form \eqref{bandsexpression} ---this is the content of Lemma \ref{NLMCLemma}. With this expression in hand, we derive an implicit equation for the constant ANMC curves. %\textcolor{red}{Then, after considering perturbations of straight cylinders of radius depending on a parameter of the form \eqref{perturbationexpression}, we will impose that these have constant ANMC equal to that of the cylinders whom they are a perturbation of. This will yield a nonlinear equation, which we will divide by a small parameter $a$ as done in the classical paper by Crandall and Rabinowitz \cite{Cr-Rab}, and afterwards consider its corresponding operator  \eqref{mainoperatorexpression}.} \textcolor{blue}{Massa detalls, no s'entén i no estic segur que sigui rellevant per a la introducció}
In Subsection \ref{FunctionSpacesSubsection} we define the functional spaces to be used on the procedure. Then, in subsections \ref{InvertibilitySubsection} and \ref{Regularitysubsection} we check the regularity and invertibility properties of the ANMC operator and its linearization needed to apply the implicit function theorem, from which Theorem \ref{mainresultperturbative} then follows.

\subsection{On the shape of minimizers.} A natural question that one may ask is whether the constant ANMC near-cylinders of Theorem \ref{mainresultperturbative}, which are critical points of \eqref{def:PANP}, are in fact volume-constrained minimizers of $\mathcal{P}_K$ for some range of volumes. Answering this question would likely involve constructing a calibration for $\mathcal{P}_K$, which is not an easy task.

In the local case, in which ODE techniques can be used, it is known (see \cite{PedrosaRitore1999} for the following claims) that for $3\leq n \leq 8$ the only volume-constrained minimizers of the classical (isotropic) perimeter are either balls or straight cylinders, depending on whether the prescribed volume $\omega$ is smaller or larger than some critical volume $\omega^*$. In contrast, when $n\geq 10$ there exist Delaunay unduloids that are volume-constrained minimizers for certain volumes. The problem remains open for $n=9$.

A simpler question is whether in the nonlocal setting these Delaunay near-cylinders are stable (instead of volume-constrained minimizers) with respect to volume-preserving variations. In the local case, unduloids are known to be unstable in dimensions $3\leq n\leq 8$ (see \cite{PedrosaRitore1999}), and, of course, as a consequence of the previous paragraph, there exist stable unduloids in dimensions $n\geq 10$. In the nonlocal setting, when $K(x)=\abs{x}^{-n-\alpha}$, Cabré, Csató and Mas \cite{CabreCsatoMas_stab} proved that for every dimension $n \geq 2$ and every $\alpha\in (0,1)$ there exists a radius $R_{\alpha}>0$ such that a straight cylinder $\{(x_1,x')\in \mathbb{R}^n : \abs{x'}<R\}$ is stable if, and only if, $R\geq R_{\alpha}$. Moreover, they show that
\begin{equation}
    \lim_{\alpha\uparrow 1} R_{\alpha}^2 = n-2.
\end{equation}

For $n=2$, and still with $K(x)=\abs{x}^{-2-\alpha}$, the stability of the near-cylinders of Theorem \ref{mainresultperturbative} is studied in a forthcoming paper by the second author \cite{Bruera_stab}. In dimension 2, nonlocal Delaunay near-cylinders are found to be unstable with respect to volume-preserving variations when $\alpha$ is close to 1 (except for some stable straight cylinders), in accordance with the local behaviour in low dimensions. The question remains open in dimensions $n\geq 3$, although in view of the discussion in the previous paragraph, we expect Delaunay near-cylinders to be unstable, at least when $\alpha$ is close to 1 and $n$ is not too large.

\subsection{The periodic nonlocal perimeter.}\label{subsec::the_periodic_nonlocal_perimeter}

In this subsection we derive expression \eqref{def:PANP}, and justify why this is the appropriate functional to study the periodic isoperimetric problem and not the relative perimeter \eqref{eq:per_K_decomposed} below.

Historically, the introduction of the nonlocal perimeter allowed the theory of minimal surfaces to be generalized to the nonlocal setting. Given an open set $S \subset \mathbb{R}^n$, a minimizing anisotropic nonlocal minimal surface $E\subset \mathbb{R}^n$ with respect to $S$ is a minimizer of the anisotropic nonlocal Dirichlet perimeter relative to $S$, namely,
\begin{equation}\label{eq:per_K_decomposed}
    \mathrm{Per}_K[F;S] := \left\{\int_{F\cap S}\int_{F^c\cap S}  + \int_{F\cap S}\int_{F^c\cap S^c} + \int_{F\cap S^c}\int_{F^c\cap S}\right\} K(x-y)\dd{x}\dd{y},
\end{equation}
when computed among sets $F\subset \mathbb{R}^n$ having the same exterior data as $E$, i.e., $E\setminus S = F\setminus S$. Notice that, if $\mathrm{Per}_K[E]<+\infty$, then $E$ is also a minimizer of \eqref{full_NL_perimeter} among competitors $F\subset \mathbb{R}^n$ having the same exterior data as $E$. Indeed, $\mathrm{Per}_K[F;S]$ and $\mathrm{Per}_K[F]$ are related by the identity
\begin{equation}\label{eq:per_K_identity}
    \mathrm{Per}_K[F] = \mathrm{Per}_K[F;S] + \int_{F\cap S^c}\int_{F^c\cap S^c}K(x-y)\dd{x}\dd{y}, 
\end{equation}
and the last term in the right hand side of \eqref{eq:per_K_identity} is the same for all competitors $F$.

Figure \ref{fig1} below includes a representation of a possible minimizer $E$ (the thick solid orange line) of $\mathrm{Per}[\ \cdot \ ; S]$ with $S=\Omega := \{x\in \mathbb{R}^n : -\pi < x_1 < \pi\}$, and of a possible competitor (the thin dashed orange line, which coincides with $E$ outside of $\Omega$). The full perimeter, $\mathrm{Per}_K[E]$, includes the interactions between all four pairs of points: $A$ and $B$, $A$ and $C$, $A'$ and $C'$, and $A'$ and $B'$, each pair corresponding to one of the double integrals resulting from putting \eqref{eq:per_K_decomposed} and \eqref{eq:per_K_identity} together. In contrast, $\mathrm{Per}[\ \cdot \ ; S]$ does not take into account the interaction between $A'$ and $B'$, since it is the same for $E$ and its competitor.

\definecolor{qqqqff}{rgb}{0,0,1}
\definecolor{ffqqqq}{rgb}{1,0,0}
\definecolor{uququq}{rgb}{0.25098039215686274,0.25098039215686274,0.25098039215686274}
\definecolor{taronja}{rgb}{1,0.64,0}
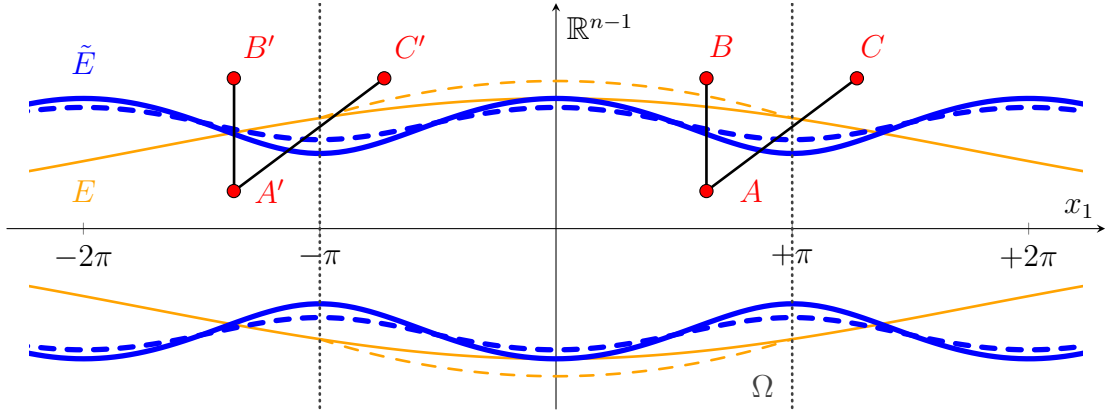
\begin{figure}[h!]\label{fig1}
\centering
\begin{tikzpicture}[line cap=round,line join=round,>=triangle 45,x=1cm,y=1cm]
\begin{axis}[
x=1cm,y=1cm,
axis lines=middle,
xmin=-7.3,
xmax=7.3,
ymin=-2.4,
ymax=3.,
xtick={-6.283185307179586,-3.14,0,3.14,6.283185307179586},
xticklabels={$-2\pi$,$-\pi$,0,$+\pi$,$+2\pi$},
xlabel=$x_1$,
ytick={-35},
ylabel = $\mathbb{R}^{n-1}$]
\clip(-7,-2.5) rectangle (7,3);

\draw[line width=1pt,color=taronja,smooth,samples=100,domain=-9.976540483050782:10.36247909422196] plot(\x,{0+sqrt(exp(-(\x)^2/30+1.1))});
\draw[line width=1pt,color=taronja,smooth,samples=100,domain=-9.976540483050782:10.36247909422196] plot(\x,{0-sqrt(exp(-(\x)^2/30+1.1))});

\draw[line width=2pt,color=qqqqff,smooth,samples=100,domain=-9.976540483050782:10.36247909422196] plot(\x,{sqrt(cos(((\x))*180/pi)+2)});
\draw[line width=2pt,color=qqqqff,smooth,samples=100,domain=-9.976540483050782:10.36247909422196] plot(\x,{0-sqrt(cos(((\x))*180/pi)+2)});

\draw[line width=2pt,dash pattern=on 5pt off 5pt,color=qqqqff,smooth,samples=100,domain=-9.976540483050782:10.36247909422196] plot(\x,{sqrt(0.6*cos(\x*180/pi)+2)});
\draw[line width=2pt,dash pattern=on 5pt off 5pt,color=qqqqff,smooth,samples=100,domain=-9.976540483050782:10.36247909422196] plot(\x,{0-sqrt(0.6*cos(\x*180/pi)+2)});

\draw[line width=1pt,dash pattern=on 5pt off 5pt,color=taronja,smooth,samples=100,domain=-3.1416:3.1416] plot(\x,{0-(\x+3.1416)*(\x-3.1416)/20+sqrt(exp(-(3.1416)^2/30+1.1))});
\draw[line width=1pt,dash pattern=on 5pt off 5pt,color=taronja,smooth,samples=100,domain=-3.1416:3.1416] plot(\x,{0(\x+3.1416)*(\x-3.1416)/20-sqrt(exp(-(3.1416)^2/30+1.1))});

\draw [line width=1pt,dash pattern=on 0.5pt off 2pt,color=uququq] (3.14,-3) -- (3.14,3);
\draw [line width=1pt,dash pattern=on 0.5pt off 2pt,color=uququq] (-3.14,-3) -- (-3.14,3);
\draw [line width=1pt] (2,0.5)-- (4,2);
\draw [line width=1pt] (2,0.5)-- (2,2);
\draw [line width=1pt] (-4.28
,0.5)-- (-2.28,2);
\draw [line width=1pt] (-4.28,0.5)-- (-4.28,2);

\begin{scriptsize}
\draw[color=taronja] (-6.28,0.5) node {$E$};
\draw[color=qqqqff] (-6.28,2.25) node {$\tilde{E}$};
\draw[color=uququq] (2.75,-2.1) node {$\Omega$};

\draw [fill=ffqqqq] (2,0.5) circle (2.5pt);
\draw[color=ffqqqq] (2.592516559084059,0.5) node {$A$};
\draw [fill=ffqqqq] (2,2) circle (2.5pt);
\draw[color=ffqqqq] (2.2011052351412597,2.4238263741845496) node {$B$};
\draw [fill=ffqqqq] (4,2) circle (2.5pt);
\draw[color=ffqqqq] (4.212154361882835,2.4238263741845496) node {$C$};
\draw [fill=ffqqqq] (-4.283185307179586,0.5) circle (2.5pt);
\draw[color=ffqqqq] (-3.80627611691186,0.5) node {$A'$};
\draw [fill=ffqqqq] (-4.283185307179586,2) circle (2.5pt);
\draw[color=ffqqqq] (-3.946306300011962,2.4238263741845496) node {$B'$};
\draw [fill=ffqqqq] (-2.2831853071795862,2) circle (2.5pt);
\draw[color=ffqqqq] (-1.9352571732703876,2.4238263741845496) node {$C'$};
\end{scriptsize}
\end{axis}
\end{tikzpicture}
\caption{Pairs of interacting points in the definition of the anisotropic nonlocal perimeter of periodic (thick blue line) and non-periodic sets (thin orange line). The dotted vertical lines mark the boundary of the slab $\Omega=\{x\in \mathbb{R}^n : -\pi < x_1 <\pi\}$. The thick dashed blue and thin dashed orange lines correspond to possible competitors in the minimization of $\mathrm{Per}_K[\ \cdot \ ;\Omega]$ and $\mathcal{P}_K[\cdot]$, respectively.}
\end{figure}

In this paper, however, we are interested in the minimization problem with \textit{periodic} exterior condition. In other words, we now do not fix the exterior datum, as opposed to the Dirichlet problem described above. Therefore, it does not make sense to seek minimizers of the (anisotropic) nonlocal perimeter relative to the slab $\Omega:=\{x\in \mathbb{R}^n:-\pi<x_1<\pi\}$. Instead, let us argue as follows. Let $\tilde{F}\subset \mathbb{R}^n$ be an open set that is $2\pi$-periodic in the $x_1$-direction, i.e., such that $\tilde{F}+2\pi k e_1 = \tilde{F}$ for every $k\in \mathbb{Z}$, with $e_1=(1,0,\dots,0)$. Then, formally,
\begin{align}
    \mathrm{Per}_K[\tilde{F}]&=\int_{\tilde{F}} \int_{\tilde{F}^c} K(x-y)\dd{x}\dd{y} = \sum_{k\in \mathbb{Z}} \int_{\tilde{F}\cap \Omega + 2\pi k e_1} \int_{\tilde{F}^c} K(x-y)\dd{x}\dd{y} \\
    &= \sum_{k\in \mathbb{Z}} \int_{\tilde{F}\cap \Omega} \int_{\tilde{F}^c - 2\pi k e_1} K(x-y)\dd{x}\dd{y} = \sum_{k\in \mathbb{Z}} \int_{\tilde{F}\cap \Omega} \int_{\tilde{F}^c } K(x-y)\dd{x}\dd{y}. \label{sum_perimeter_periodic}
\end{align}
Therefore, we must minimize the functional
\begin{equation}
    \mathcal{P}_K[\tilde{F}]:=\int_{\tilde{F}\cap \Omega} \int_{\tilde{F}^c } K(x-y)\dd{x}\dd{y},
\end{equation}
which is precisely the periodic anisotropic nonlocal perimeter as we have defined it in \eqref{def:PANP}. Of course, the previous computation is merely formal, since the series in \eqref{sum_perimeter_periodic} is divergent unless $\mathcal{P}_K[\tilde{F}]=0$. As we mentioned earlier in the introduction, since the sets $\tilde{F}$ do not satisfy any exterior condition, we have $\inf \mathcal{P}_K = 0$, where the infimum is taken over all open sets that are $2\pi$-periodic in the $x_1$-direction. Therefore, it becomes necessary to introduce some kind of constraint on the class of admissible competitors. We do this in the form of a volume constraint, i.e., given $\omega>0$, we require that all competitors $\tilde{F}$ satisfy
\begin{equation}
    |\tilde{F}\cap \Omega | = \omega,
\end{equation}
where $\abs{\cdot}$ denotes the $n$-dimensional Lebesgue measure.

In Figure \ref{fig1}, a possible volume-constrained minimizer $\tilde{E}$ of $\mathcal{P}_K$ is represented by the thick solid blue line. Notice that $\mathcal{P}_K$ only takes into account the interactions between points $A$ and $B$, and $A$ and $C$. Since the set $\tilde{E}$ is periodic, these are identical to the interactions between $A'$ and $B'$, and $A'$ and $C'$, respectively. Notice, moreover, that now a competitor $\tilde{F}$ (represented by the thin dashed blue line in Figure \ref{fig1}) need not coincide with $\tilde{E}$ in $\mathbb{R}^n \setminus \Omega$ (in fact, if it did, then by periodicity we would have $\tilde{F}=\tilde{E}$).

\section{Existence and symmetry of constrained minimizers of the periodic anisotropic nonlocal perimeter} \label{VASection}
In this section we prove theorems \ref{main_theorem_var} and \ref{theorem_monotonicity_symmetric_rearrangement}. First, we introduce some notation and definitions that will be needed in the proofs. Following the introduction, we denote
\begin{equation}
    \Omega = \{(x_1,x')\in \mathbb{R}^n : -\pi < x_1<\pi\}
\end{equation}
throughout the whole text.
\subsection{Cylindrical and symmetric decreasing rearrangements} \label{SubsectionNotation}

Here we introduce the concepts of cylindrical and symmetric decreasing periodic rearrangements of sets, which are crucial for the results in this section. 

Given a measurable set $E\subset \mathbb{R}^n$ and $x_1\in \mathbb{R}$, we denote by $E_{x_1}=\{x'\in \mathbb{R}^{n-1} : (x_1,x')\in E\}$ its transversal sections. 
\begin{definition}
    The \textit{cylindrical rearrangement} of $E$ is the unique set $E^{*\mathrm{cyl}}$ such that, for every $x_1\in \mathbb{R}$,
    \begin{itemize}
        \item $(E^{*\mathrm{cyl}})_{x_1}$ is an open ball of $\mathbb{R}^{n-1}$ centered at 0,
        \item $\abs{(E^{*\mathrm{cyl}})_{x_1}}=\abs{E_{x_1}}$, where $\abs{\ \cdot\ }$ denotes the $(n-1)$-dimensional Lebesgue measure,
    \end{itemize}
    with the understanding that $(E^{*\mathrm{cyl}})_{x_1}=\mathbb{R}^{n-1}$ if $\abs{E_{x_1}}=+\infty$.
\end{definition}
In other words, for every $x_1\in \mathbb{R}$, the set $(E^{*\mathrm{cyl}})_{x_1}$ is the \textit{Schwarz rearrangement} of $E_{x_1}$, which we denote by $(E_{x_1})^{*(n-1)}$, i.e.,
\begin{equation}
    (E^{*\mathrm{cyl}})_{x_1} = (E_{x_1})^{*(n-1)}.
\end{equation}

Similarly, given a set $A\subset \mathbb{R}^{n}$, for $x'\in \mathbb{R}^{n-1}$ we denote by $A_{x'}=\{x_1\in \mathbb{R} : (x_1,x')\in A\}$ its longitudinal sections. If $A_{x'}$ has finite length for every $x'\in \mathbb{R}^{n-1}$, we define the set $A^*$ as the Steiner symmetrization of $A$ with respect to $\{x_1=0\}$. That is, $A^*\subset \mathbb{R}^n$ is the unique set satisfying the following properties.
\begin{itemize}
    \item $A^*$ is symmetric with respect to the hyperplane $\{x_1=0\}$,
    \item $\abs{(A^*)_{x'}}=\abs{A_{x'}}$ for every $x'\in \mathbb{R}^{n-1}$, where $\abs{\cdot}$ denotes the 1-dimensional Lebesgue measure.
\end{itemize}

Let $F\subset \mathbb{R}^n$ be a measurable set that is $2\pi$-periodic in the first coordinate.
\begin{definition}
The \textit{symmetric decreasing periodic rearrangement} of $F$, denoted $F^{*\mathrm{per}}$, is the $2\pi$-periodic extension in the $x_1$-direction of the set $(F\cap \Omega)^*$. That is,
\begin{equation}
    F^{*\mathrm{per}} := \bigcup_{k\in \mathbb{Z}} \left\{(F\cap \Omega )^*+2\pi k e_1\right\},
\end{equation}
with $e_1=(1,0,\dots)\in \mathbb{R}^n$.
\end{definition}

Finally, we define the rearrangement of real valued functions on $\mathbb{R}$. Given a set $A\subset \mathbb{R}$ we define the symmetric decreasing rearrangement of its characteristic function $\chi_A$ by
\begin{equation}
    \chi^*_A:=\chi_{A^*}.
\end{equation}
Using the layer cake representation, given a real valued function $f:\mathbb{R}\to \mathbb{R}$ we define its symmetric decreasing rearrangement $f^*$ by\footnote{To be rigorous, the symmetric decreasing rearrangement is well defined for functions that vanish at infinity, i.e., those whose superlevel sets $\{\abs{f}>t\}$ have finite measure for all $t>0$.}
\begin{equation}
    f^*(x):=\int_0^\infty \chi^*_{\{\abs{f}>t\}} \dd{t} = \sup\{t:x\in \{\abs{f}>t\}^*\}.
\end{equation}
If $u:\mathbb{R}\to \mathbb{R}$ is $2\pi$-periodic, we define its symmetric decreasing periodic rearrangement as the unique $2\pi$-periodic function $u^{*\mathrm{per}}$ such that $u^{*\mathrm{per}}\chi_{(-\pi,\pi)}=(u\chi_{(-\pi,\pi)})^*$. Notice that if $E\subset \mathbb{R}^n$ is of the form
\begin{equation}
    E=\{x\in \mathbb{R}^n:\abs{x'}<u(x_1)\},
\end{equation}
then
\begin{equation}
    E^{*\mathrm{per}}=\{x\in \mathbb{R}^n:\abs{x'}<u^{*\mathrm{per}}(x_1)\}.
\end{equation}

\subsection{Monotonicity under cylindrical and symmetric decreasing periodic rearrangements} 
In this subsection we prove the two monotonicity results mentioned in the introduction. Namely, that, under appropriate assumptions on the kernel, $\mathcal{P}_K$ does not increase by cylindrical and symmetric decreasing periodic rearrangements. These two facts do not depend on the existence of minimizers ---in fact, as we will see in the next subsection, the proof of existence relies fundamentally on Lemma \ref{lemma_cylindrical}.

\begin{proof}[Proof of Lemma \ref{lemma_cylindrical}]
The proof is the same as \cite[Lemma 3.1]{CCM}, with the function $g_{x_1,y_1}$ replaced by
\begin{equation}
        g_{x_1,y_1}(z'):=K(x_1-y_1,z'),
\end{equation}
where $x_1,y_1\in \mathbb{R}$ and $x_1\neq y_1$.

The only assumptions needed on $g_{x_1,y_1}$ are that it be positive, radially symmetric decreasing, and that $g_{x_1,y_1}\in L^1(\mathbb{R}^{n-1})$ for almost every $(x_1,y_1)\in \mathbb{R}\times \mathbb{R}$, which are granted by hypotheses \eqref{H2::RotInv}---\eqref{H4::int}.

We next give a sketch of the proof; we refer the reader to \cite[Lemma 3.1]{CCM} for the details. We begin by writing
\begin{align}
    \mathcal{P}_K[E] &= \int_{-\pi}^\pi \dd{x_1}\int_{\mathbb{R}}\dd{y_1} \left\{\int_{E_{x_1}}\dd{x'}\int_{\mathbb{R}^{n-1}}\dd{y'} - \int_{E_{x_1}}\dd{x'}\int_{E_{y_1}}\dd{y'}\right\} K(x_1-y_1,x'-y')\\ 
    &= \begin{multlined}[t][0.75\displaywidth]
        \int_{-\pi}^\pi \dd{x_1}\int_{\mathbb{R}}\dd{y_1} \left\{\int_{E_{x_1}}\dd{x'}\int_{\mathbb{R}^{n-1}}\dd{z'}g_{x_1,y_1}(z') \right. \\ \left. - \int_{E_{x_1}}\dd{x'}\int_{E_{y_1}}\dd{y'}g_{x_1,y_1}(x'-y')\right\} 
    \end{multlined}\\
    &= \begin{multlined}[t][0.75\displaywidth]
        \int_{-\pi}^\pi \dd{x_1}\int_{\mathbb{R}}\dd{y_1} \left\{\abs{E_{x_1}}\int_{\mathbb{R}^{n-1}}\dd{z'}g_{x_1,y_1}(z') \right. \\ \left. - \int_{E_{x_1}}\dd{x'}\int_{E_{y_1}}\dd{y'}g_{x_1,y_1}(x'-y')\right\} 
    \end{multlined}\\
    &\geq \begin{multlined}[t][0.75\displaywidth]
        \int_{-\pi}^\pi \dd{x_1}\int_{\mathbb{R}}\dd{y_1} \left\{\abs{\left(E_{x_1}\right)^{*(n-1)}}\int_{\mathbb{R}^{n-1}}\dd{z'}g_{x_1,y_1}(z') \right. \\ \left. - \int_{(E_{x_1})^{*(n-1)}}\dd{x'}\int_{(E_{y_1})^{*(n-1)}}\dd{y'}g_{x_1,y_1}(x'-y')\right\}
    \end{multlined}\\ 
    &= \mathcal{P}_K[E^{*\mathrm{cyl}}],
\end{align}
where in the last inequality we have used the facts that $\abs{(E_{x_1})^{*(n-1)}} = \abs{E_{x_1}}$ by definition, and 
\begin{equation}\label{riesz_rearrangement_applied_to_g}
    \int_{E_{x_1}}\dd{x'}\int_{E_{y_1}}\dd{y'}g_{x_1,y_1}(x'-y') \leq \int_{(E_{x_1})^{*(n-1)}}\dd{x'}\int_{(E_{y_1})^{*(n-1)}}\dd{y'}g_{x_1,y_1}(x'-y')
\end{equation}
by Riesz's rearrangement inequality applied with $g_{x_1,x_2}$ and the characteristic functions of $E_{x_1}$ and $E_{y_1}$. The characterization of the case of equality follows by using the strict version of Riesz's inequality, which essentially states that equality in \eqref{riesz_rearrangement_applied_to_g} holds if, and only if, either $\abs{E_{x_1}}\cdot \abs{E_{y_1}}=0$, where $\abs{\cdot}$ denotes the $(n-1)$-dimensional Lebesgue measure, or $E_{x_1}$ and $E_{y_1}$ are both open balls.
\end{proof}

Next, we prove Theorem \ref{theorem_monotonicity_symmetric_rearrangement}, i.e., that $\mathcal{P}_K$ does not increase by symmetric decreasing periodic rearrangements. The proof of this result, which is based on the following Riesz rearrangement inequality on the circle, follows the ideas in \cite[Theorem 1.2]{CCM} and \cite[Theorem 1.4]{PerIntDiffEq}. 
\begin{theorem}[{\cite[Theorem 2]{BaernsteinTaylor1976}, \cite[Theorem 1]{FriedbergLuttinger1976}}]\label{TheoremRearrangementIneqCircle}
    Let $f,g,h:\mathbb{R}\to \mathbb{R}$ be nonnegative $2\pi$-periodic measurable functions. Assume that $g$ is even, as well as nonincreasing in $(0,\pi)$. 

    Then,
    \begin{equation}\label{RearrangementIneqCircle}
        \int_{-\pi}^\pi \dd{x} \int_{-\pi}^\pi \dd{y} f(x)g(x-y)h(y)\leq \int_{-\pi}^\pi \dd{x} \int_{-\pi}^\pi \dd{y} f^{*\mathrm{per}}(x)g(x-y)h^{*\mathrm{per}}(y).
    \end{equation}

    In addition, if $g$ is decreasing in $(0,\pi)$ and the left hand side of \eqref{RearrangementIneqCircle} is finite, then, equality holds in \eqref{RearrangementIneqCircle} if, and only if, either of the following holds:
    \begin{enumerate}[(i)]
        \item Either $f$ or $h$ is constant almost everywhere.
        \item There exists $z\in \mathbb{R}$ such that $f(x)=f^{*\mathrm{per}}(x+z)$ and $h(x)=h^{*\mathrm{per}}(x+z)$ for almost every $x\in \mathbb{R}$.
    \end{enumerate}
\end{theorem}
\begin{proof}[Proof of Theorem \ref{theorem_monotonicity_symmetric_rearrangement}]
    We first show that $\mathcal{P}_K[E^{\ast \mathrm{per}}] \le \mathcal{P}_K[E]$. Using Fubini's theorem to exchange the order of integration, we write
    \begin{align}
    \mathcal{P}_{K}[E] 
    &=\int_{\mathbb{R}^{n-1}} \dd{x'} \int_{\mathbb{R}^{n-1}} \dd{y'} \left\{\int_{E_{x'}\cap (-\pi,\pi)} \dd{x_1}  \int_{\mathbb{R}} \dd{y_1} - \int_{E_{x'}\cap (-\pi,\pi)} \dd{x_1}  \int_{ E_{y'}} \dd{y_1} \right\} K(x-y)\\
    &
    \begin{multlined}[b][0.82\displaywidth]
        =\int_{\mathbb{R}^{n-1}} \dd{x'} \int_{\mathbb{R}^{n-1}} \dd{y'} \left\{\int_{E_{x'}\cap (-\pi,\pi)} \right. \dd{x_1} \int_{\mathbb{R}} \dd{y_1} K(x-y)\\ \left. - \sum_{k\in \mathbb{Z}}\int_{E_{x'}\cap (-\pi,\pi)} \dd{x_1}  \int_{ E_{y'}\cap (-\pi, \pi )} \dd{y_1} K(x-y-2\pi ke_1)\right\}.
    \end{multlined}\label{eq::expression_GK_longitudinal}
    \end{align}
    Now, for almost every $(x',y')\in \mathbb{R}^{n-1}\times \mathbb{R}^{n-1}$, changing variables to $s = x_1-y_1$, we get
    \begin{align}
    &\begin{multlined}[t][0.7\displaywidth]
    \int_{E_{x'}\cap (-\pi,\pi)} \dd{x_1} \int_{\mathbb{R}} \dd{y_1}K(x_1-y_1, x'-y') \\ - \sum_{k\in \mathbb{Z}}\int_{E_{x'}\cap (-\pi,\pi)} \dd{x_1}  \int_{ E_{y'}\cap (-\pi, \pi )} \dd{y_1}  K(x_1-y_1-2\pi k, x'-y') 
    \end{multlined}\\ 
    &=
    \abs{E_{x'}\cap (-\pi,\pi)}  \int_{\mathbb{R}} \dd{s} K(s, x'-y')  - \int_{E_{x'}\cap (-\pi,\pi)} \dd{x_1}  \int_{ E_{y'}\cap (-\pi, \pi )} \dd{y_1} g_{x',y'}(x_1-y_1),\label{eq::expression_longitudinal_section}
    \end{align}
    where we have defined
    \begin{equation}
        g_{x',y'}(s) := \sum_{k \in \mathbb{Z}} K(s+2\pi k, x'-y').
    \end{equation}
    Notice that, by \eqref{H5::UpperBound}, $g_{x',y'}$ is well defined whenever $x'\neq y'$, and it is nonnegative, $2\pi$-periodic, and even. Moreover, by Lemma \ref{lemma_monotonicity_g} below, it is nonincreasing in $(0,\pi)$ whenever $K$ satisfies one of the conditions \textit{(1)--(3)}.
    
    Plugging \eqref{eq::expression_longitudinal_section} back into \eqref{eq::expression_GK_longitudinal}, and using Theorem \ref{TheoremRearrangementIneqCircle} applied with $f=\chi_{E_{x'}}$ and $h=\chi_{E_{y'}}$ ---which are $2\pi$-periodic since $E$ is, too--- we obtain
    \begin{align}
        \mathcal{P}_K[E]&=\int_{\mathbb{R}^{n-1}} \dd{x'} \int_{\mathbb{R}^{n-1}} \dd{y'} \begin{multlined}[t]
        \left\{\abs{E_{x'}\cap (- \pi, \pi)}  \int_{\mathbb{R}} \dd{s} K(s, x'-y')\right. \\ \left.  - \int_{E_{x'}\cap (-\pi,\pi)} \dd{x_1}  \int_{ E_{y'}\cap (-\pi, \pi )} \dd{y_1} g_{x',y'}(x_1-y_1)\right\}
        \end{multlined}\\ 
        & \begin{multlined}[b][0.82\displaywidth]
        \geq \int_{\mathbb{R}^{n-1}} \dd{x'} \int_{\mathbb{R}^{n-1}} \dd{y'}
        \left\{\abs{(E^{*\mathrm{per}})_{x'} \cap (- \pi, \pi)}  \int_{\mathbb{R}} \dd{s} K(s, x'-y')\right. \\ \left.  - \int_{(E^{*\mathrm{per}})_{x'}\cap (-\pi,\pi)} \dd{x_1}  \int_{ (E^{*\mathrm{per}})_{y'}\cap (-\pi, \pi )} \dd{y_1} g_{x',y'}(x_1-y_1)\right\}
        \end{multlined}\label{eq::inequality_sym_per_rearr}\\ 
        &=\mathcal{P}_K[E^{{*\mathrm{per}}}],
    \end{align}
    where we have used the fact that, by definition, $\abs{E_{x'} \cap (- \pi, \pi)}  =  \abs{(E^{*\mathrm{per}})_{x'} \cap (- \pi, \pi)}$.

    The proof of the second part of the theorem is analogous to that of \cite[Theorem 1.2]{CCM}, but we include it for completeness. We have to prove that if $K$ satisfies any of the conditions \textit{(1)'--(3)'} and $\mathcal{P}_K[E^{\ast \mathrm{per}}] = \mathcal{P}_K[E]<+\infty$, then $E=E^{\ast \mathrm{per}}+ce_1$ for some $c\in\mathbb{R}$, up to a set of measure zero. We first point out that the second part of Lemma \ref{lemma_monotonicity_g} yields that $g$ is decreasing whenever $K$ satisfies one of the conditions \textit{(1)'--(3)'}. From \eqref{eq::inequality_sym_per_rearr} we deduce that if $\mathcal{P}_K[E^{\ast \mathrm{per}}] = \mathcal{P}_K[E]$, then we must have
    \begin{multline} \label{MSRThmAux4}
        \int_{E_{x'}\cap(-\pi,\pi)} \dd{x_1} \int_{E_{y'}\cap(-\pi,\pi)} \dd{y_1} g_{x',y'}(x_1-y_1)= \\ = \int_{(E^{\ast \mathrm{per}})_{x'}\cap(-\pi,\pi)} \dd{x_1} \int_{(E^{\ast \mathrm{per}})_{y'}\cap(-\pi,\pi)} \dd{y_1} g_{x',y'}(x_1-y_1)
    \end{multline}
    for almost every $(x',y')\in \mathbb{R}^{n-1}\times \mathbb{R}^{n-1}$.

    Let us now consider the set $J\subset \mathbb{R}^{n-1}$ of points $x' \in \mathbb{R}^{n-1}$ such that $E_{x'}\cap(-\pi,\pi)$ is neither $(-\pi,\pi)$ nor $\emptyset$ (up to sets of measure zero). If $J$ has null $(n-1)$-dimensional Lebesgue measure, then it is easy to check that $E=E^{\ast \mathrm{per}}$ up to a set of measure zero in $\mathbb{R}^n$. Therefore, we may assume that $J$ has positive measure. Then $J\times J$ has positive measure in $\mathbb{R}^{n-1}\times\mathbb{R}^{n-1}$, and \eqref{MSRThmAux4} holds for almost every $(x',y')\in J\times J$. By definition of $J$, for every such $(x',y')$ the functions $f=\chi_{E_{x'}}$ and $h=\chi_{E_{y'}}$ are non-constant, and, therefore, by the characterization of equality in Theorem \ref{TheoremRearrangementIneqCircle} applied to \eqref{MSRThmAux4} at points $x',y'\in J$, for almost every $(x',y')\in J\times J$ we must have
    \begin{equation}
        E_{x'} = (E^{\ast \mathrm{per}})_{x'} + c(x',y') \; \text{ and } \; E_{y'} = (E^{\ast \mathrm{per}})_{y'} + c(x',y')
    \end{equation}
    for some $c(x',y') \in \mathbb{R}$ depending possibly on $x'$ and $y'$. 
    
    To see that $c$ is independent of $x'$ and $y'$, we repeat the argument for two different $x',\overline{x}' \in J$ but with the same $y'$. This yields
    \begin{equation} \label{2.2Aux5}
        (E^{\ast \mathrm{per}})_{y'} + c(x',y') = E_{y'} = (E^{\ast \mathrm{per}})_{y'} + c(\overline{x}',y').
    \end{equation}
    Since $y' \in J$, $(E^{\ast \mathrm{per}})_{y'}$ is a $2\pi$-periodic set such that $(E^{\ast \mathrm{per}})_{y'}\cap(-\pi,\pi)$ is a nonempty open interval of $(-\pi,\pi)$. This, together with \eqref{2.2Aux5}, gives that $c(x',y')=c(\overline{x}',y')$ modulo $2\pi$, hence we can choose $c(\overline{x}',y')=c(x',y')$. Repeating the same argument for $E_{x'}$ shows that $c(x',y')$ is also independent of $y'$ and thus, $E_{y'} = (E^{\ast \mathrm{per}})_{y'} + c$ for some $c\in\mathbb{R}$ and every $y'\in J$. Therefore, by definition of $J$, $E=E^{\ast \mathrm{per}}+ce_1$ up to a set of measure zero, as we wanted to see.
\end{proof}

Next, we prove the monotonicity of the function $g_{x',y'}$ used in the proof of Theorem \ref{theorem_monotonicity_symmetric_rearrangement} above. To simplify the notation, we omit the dependence on $x'$ and $y'$ in the subscript.
\begin{lemma}\label{lemma_monotonicity_g}
Let $K:\mathbb{R}^n\to [0,+\infty)$ be a nonnegative measurable kernel satisfying the upper bound \eqref{H5::UpperBound}. Suppose, in addition, that $K$ satisfies, for almost every $x'\in \mathbb{R}^{n-1}$, at least one of the following.
\begin{enumerate}[(1)]
    \item $K(\cdot,x')$ is convex in $(0,+\infty)$,
    \item The function $\tau>0\mapsto K(\sqrt{\tau},x')$ is completely monotonic,
    \item $K(\cdot,x')$ is nonincreasing in $(0,\pi)$ and vanishes in $[\pi,+\infty)$.
\end{enumerate}
Then, for almost every $x'\in \mathbb{R}^{n-1}$, the function
\begin{equation}
    g(s) = \sum_{k\in \mathbb{Z}}K(s+2k \pi,x')
\end{equation}
is nonincreasing in $(0,\pi)$.

Assume, moreover, that for almost every $x'\in \mathbb{R}^{n-1}$, at least one of the following holds.
\begin{enumerate}[(1)']
    \item $K(\cdot,x')$ is convex in $(0,+\infty)$ and strictly convex in $(c,c+2\pi)$ for some $c\geq 0$.
    \item $K(\cdot,x')$ is as in (2) and nonzero.
    \item $K(\cdot,x')$ is decreasing in $(0,\pi)$ and vanishes in $[\pi,+\infty)$.
\end{enumerate}

Then, for almost every $x'\in \mathbb{R}^{n-1}$, the function $g$ is strictly decreasing in $(0,\pi)$.
\end{lemma}
\begin{proof}
Assumption \eqref{H5::UpperBound} and the nonnegativity of $K$ ensure that the function $g$ is well defined. To prove the monotonicity, we deal with each case separately.

\proofstep{Case 1. $K(\cdot,x')$ is convex.}
We write
\begin{equation}
    g(s) = \sum_{k\in \mathbb{Z}} K(s+2\pi k,x') = \sum_{k\geq 0} \left( K(s+2\pi k,x')+K(2\pi (k+1) - s,x')\right).
 \end{equation} 
Now, fix $s\in (0,\pi)$. Then, for every $\varepsilon>0$ such that $s+\varepsilon<\pi$ we have
\begin{align}
    g(s+\varepsilon)-g(s) = \sum_{k\geq 0} &\left\{  \left[K(s + \varepsilon+2\pi k,x') - K(s +2\pi k,x')\right]\right. \\ 
    &-\left.\left[ K(2\pi (k+1) - s,x')-K(2\pi (k+1) - s- \varepsilon,x')\right]\right\}. \label{g(s+eps)-g(s)}
\end{align}
Then, since 
\begin{equation}
    2\pi k + s < 2\pi k + s + \varepsilon \leq 2 \pi (k+1) - s - \varepsilon < 2 \pi (k+1) - s
\end{equation}%A la segona igualtat és 2\pi k + s + \varepsilon?
for all $s\in (0,\pi)$ and all $k\geq 0$, by the convexity of $K(\cdot,x')$ we deduce that $g(s+\varepsilon)\leq g(s)$, and hence, since $\varepsilon>0$ was arbitrary, that $g$ is nonincreasing.

We claim that, if $K(\cdot,x')$ is convex in $(0,+\infty)$ and strictly convex in $(c,c+2\pi)$ for some $c\geq 0$, for $k=[c/2\pi]$, where $[\cdot]$ denotes the nearest integer\footnote{We use the convention that $[\xi]=k\in \mathbb{Z}$ if $\xi\in [k-1/2,k+1/2)$.}, the $k$-th summand in \eqref{g(s+eps)-g(s)} is negative, while the rest are nonpositive, and hence $g(s+\varepsilon)<g(s)$. 

To show this, let us denote, for a fixed $\varepsilon>0$,
\begin{equation}
    d(s) = K(s+\varepsilon,x')-K(s,x').
\end{equation}
Notice that, if $K(\cdot,x')$ is strictly convex in some neighborhood of a point $s$, then $d$ is strictly increasing in that same neighborhood. Let us assume for simplicity that $[c/2\pi]=1$, i.e., that $\pi\leq c<3\pi$. We must show that
\begin{equation}
    d(s+2\pi)-d(4\pi - s-\varepsilon)<0.
\end{equation}
We distinguish three cases.
\begin{itemize}
    \item If $c<s+2\pi<c+2\pi<4\pi-s - \varepsilon$, then,
    \begin{equation}
        d(s+2\pi)-d(4\pi-s-\varepsilon) =d(s+2\pi)-d(c+2\pi)+d(c+2\pi)-d(4\pi-s-\varepsilon)<0
    \end{equation}
    by the strict convexity of $K$ in $(c,c+2\pi)$.
    \item If $c<s+2\pi<4\pi-s-\varepsilon<c+2\pi$, the claim follows by the strict convexity of $K$ in $(c,c+2\pi)$.
    \item If $s+2\pi<c<4\pi-s-\varepsilon<c+2\pi$, then,
    \begin{equation}
        d(s+2\pi)-d(4\pi-s-\varepsilon) =d(s+2\pi)-d(c)+d(c)-d(4\pi-s-\varepsilon)<0
    \end{equation}
    by the strict convexity of $K$ in $(c,c+2\pi)$.
\end{itemize}

\proofstep{Case 2. $\tau>0 \mapsto K(\sqrt{\tau},x')$ is completely monotonic.}
If $K$ is identically zero, then so is $g$ and we are done. Otherwise, by Bernstein's theorem on completely monotonic functions, the function $\tau\mapsto K(\sqrt{\tau},x')$ can be written, for each $x'\in \mathbb{R}^{n-1}$, as the Laplace transform of a nonnegative Radon measure, i.e.,
\begin{equation}
    K(s,x') = \int_{[0,\infty)} e^{-r s^2}\mu_{x'}(\dd{r}).
\end{equation}
Then,
\begin{equation}
    g(s) = \sum_{k\in \mathbb{Z}}\int_{[0,\infty)} e^{-r(s+2\pi k)^2}\mu_{x'}(\dd{r}) = \int_{[0,\infty)} G_r(s) \mu_{x'}(\dd{r}),
\end{equation}
where
\begin{equation}
    G_r(s) = \sum_{k\in \mathbb{Z}} e^{-r(s+2\pi k)^2}.
\end{equation}

We now claim that $G_r$ is decreasing in $(0,\pi)$ for every $r>0$, and hence so is $g$. This follows immediately from the fact that the fundamental solution of the heat equation in $(-\pi,\pi)$ with periodic boundary conditions,
\begin{equation}
    \Lambda_t(s) = \frac{1}{\sqrt{4\pi t}}\sum_{k\in \mathbb{Z}} e^{-\frac{(s+2\pi k)^2}{4t}},
\end{equation}
is decreasing in $(0,\pi)$ for all $t>0$ (see, for example, \cite[Theorem B.1]{CCM}).

\proofstep{Case 3. $K(\cdot,x')$ is nonincreasing in $(0,\pi)$ and vanishes in $[\pi,+\infty)$.}

Clearly in this case $g(s) = K(s,x')$ for all $s\in (0, \pi)$, and therefore the (strict) monotonicity of $g$ follows from the (strict) monotonicity of $K(\cdot,x')$.
\end{proof}

%Coment el contraexemple de g no decreasing
\begin{comment}
\textcolor{red}{Our conditions on $K$ are not sufficient to prove that $g$ is decreasing in $(0,\pi)$:} Consider the kernel $K:\mathbb{R}^2 \to (0,+\infty)$ in the plane given by
    
    $$K(x_1,x_2) = (1+ (A-1) (1-\cos(x_1))) |x|^{-2-\alpha},$$
    where $\alpha \in (0,1)$ and

    $$A := \frac{1}{2} + \frac{\sum_{k \in \mathbb{Z}} ((2\pi k)^2+1)^{-(2+\alpha)/2}}{\sum_{k \in \mathbb{Z}} ((\pi + 2\pi k)^2+1)^{-(2+\alpha)/2}}\ge 1.$$

    It is immediate to check that $K$ satisfies \eqref{H1::Symetry}-\eqref{H5::LowerBound} and \eqref{H5::UpperBound}. Moreover, the function
    
    $$g(z) = \sum_{k \in \mathbb{Z}} K(z+2\pi k, 1) = \sum_{k \in \mathbb{Z}} (1+ (A-1) (1-\cos(z))) \left((z+2\pi k)^2+1\right)^{-(2+\alpha)/2}$$
    is continuous and satisfies

    $$g(0) = \sum_{k \in \mathbb{Z}} ((2\pi k)^2+1)^{-(2+\alpha)/2}$$
    and

    \begin{align}
        g(\pi) &= (2A - 1) \sum_{k \in \mathbb{Z}} ((\pi + 2\pi k)^2+1)^{-(2+\alpha)/2}\\
        &= 2 \frac{\sum_{k \in \mathbb{Z}} ((2\pi k)^2+1)^{-(2+\alpha)/2}}{\sum_{k \in \mathbb{Z}} ((\pi + 2\pi k)^2+1)^{-(2+\alpha)/2}} \sum_{k \in \mathbb{Z}} ((\pi + 2\pi k)^2+1)^{-(2+\alpha)/2}\\
        &= 2 \sum_{k \in \mathbb{Z}} ((2\pi k)^2+1)^{-(2+\alpha)/2} = 2 g(0),
    \end{align}
    which means that $g(0) < g(\pi)$ and thus $g$ can not possibly be decreasing in $(0,\pi)$.
\end{comment}

\subsection{Existence of minimizers}
The proof of existence uses the direct method of the Calculus of Variations, the monotonicity by cylindrical rearrangements proved in the previous subsection, and the compact embedding of $W^{s,1}(-\pi,\pi)$ in $L^1(-\pi,\pi)$.

To use the latter, we need a preliminary lemma, which was proved by Cabré, Csató, and Mas in \cite{CCM}. Let us denote by $\mathcal{P}_\alpha$ the functional $\mathcal{P}_K$ corresponding to the homogeneous isotropic kernel, $K(x)={\abs{x}^{-n-\alpha}}$.
\begin{lemma}[{\hspace{1sp}\cite[Lemma 2.2 and Theorem 1.1 (i)]{CCM}}]\label{lemma_bound_on_ws1}
Let $\alpha\in (0,1)$, $n\geq 2$, and $E\subset \mathbb{R}^n$ be a set of the form
\begin{equation}
    E=\{x\in \mathbb{R}^n : \abs{x'}<u(x_1)\},
\end{equation}
with $u:\mathbb{R}\to \mathbb{R}$ measurable, nonnegative, and $2\pi$-periodic. 

Then, there exists a constant $C_{n,\alpha}$, depending only on $n$ and $\alpha$, such that 
\begin{equation}
    \left[u^{n-1}\right]_{W^{\alpha,1}(-\pi,\pi)}\leq C_{n,\alpha} \left(1+ \mathcal{P}_\alpha [E]\right),
\end{equation}
where $[\cdot]_{W^{\alpha,1}(-\pi,\pi)}$ denotes the $W^{\alpha,1}(-\pi,\pi)$-seminorm.
\end{lemma}

\begin{proof}[Proof of Theorem \ref{main_theorem_var}]
We prove each claim separately.

\proofstep{Proof of (i).} The proof follows the same argument as \cite[Theorem 1.1 (i)]{CCM}. Notice first that, by \eqref{H5::LowerBound}, we have
\begin{equation}\label{inequality_ps_gs}
    \mathcal{P}_{\alpha}[E]\leq \frac{1}{\lambda} \mathcal{P}_K[E]
\end{equation}
for every measurable set $E\subset \mathbb{R}^n$. Now, let $\{E_k\}$ be a minimizing sequence for $\mathcal{P}_K$ among sets with volume $\omega$ in $\Omega$. By Lemma~\ref{lemma_cylindrical} we may assume that every $E_k$ is of the form
\begin{equation}
    E_k = \{x\in \mathbb{R}^n : \abs{x'}< u_k(x_1)\}
\end{equation}
for some nonnegative measurable $2\pi$-periodic function $u_k$. Using Lemma \ref{lemma_bound_on_ws1} and \eqref{inequality_ps_gs} we deduce that
\begin{equation}
    \left[u^{n-1}_k\right]_{W^{\alpha,1}(-\pi,\pi)} \leq C_{n,\alpha}(1+\mathcal{P}_{\alpha}[E_k])\leq C_{n,\alpha}\left(1+ \frac{1}{\lambda}\mathcal{P}_K[E_k]\right) \leq C.
\end{equation}
Moreover, because of the volume constraint we have
\begin{equation}
    \abs{B^{n-1}_1}\int_{-\pi}^\pi u_{k}(x_1)^{n-1} \dd{x_1} =\omega,
\end{equation}
where $B^{n-1}_1$ denotes the open unit ball of $\mathbb{R}^{n-1}$.

By compactness of the embedding $W^{\alpha,1}(-\pi, \pi)\hookrightarrow~L^1(-\pi,\pi)$, there exists a subsequence of $\{u^{n-1}_k\}$ converging strongly in $L^1(-\pi,\pi)$ to a nonnegative function $v\in W^{\alpha,1}(-\pi,\pi)$. Then, denoting $u=v^{1/(n-1)}$, we also have
\begin{equation}
    \abs{B^{n-1}_1}\int_{-\pi}^\pi u(x_1)^{n-1} \dd{x_1} = \omega.
\end{equation}
Extending $u$ to be $2\pi$-periodic in $\mathbb{R}$ we obtain an admissible candidate for minimizing $\mathcal{P}_K$:
\begin{equation}
    E:=\{x\in \mathbb{R}^n : \abs{x'}<u(x_1)\}.
\end{equation}

It only remains to see that $\mathcal{P}_K$ is lower semicontinuous, since, then,
\begin{equation}
    \mathcal{P}_K[E] \leq \liminf_{k} \mathcal{P}_K[E_k] = \inf \mathcal{P}_K.
\end{equation}
This follows easily from the fact that $K$ is nonnegative and Fatou's lemma.

\proofstep{Proof of (ii).} The claim on the symmetry and the monotonicity of $u$ is an immediate consequence of Theorem \ref{theorem_monotonicity_symmetric_rearrangement}.

\proofstep{Proof of (iii).} Let $x^0\in \partial E$ be such that $u$ is $C^{1,\beta}$ in a neighborhood $U(x_1^0)$ of $x_1^0$ for some $\beta\in (\alpha,1)$, and $u(x_1^0)>0$. 

We will show first that $H_K[E]$ is constant in a neighborhood of $x^0$. By the periodicity of $E$, we may assume that $x_1^0\in (-\pi, \pi)$. We take any smooth function $\xi=\xi(x_1)$ compactly supported in $U(x_1^0)\cap \{-\pi<x_1<\pi\}$, and extend it to be $2\pi$-periodic. We can calculate the first variation of the volume functional
\begin{align}
    \derivative{}{\varepsilon}\Big\vert_{\varepsilon=0} &\abs{B^{n-1}_1}\int_{-\pi}^\pi \dd{x_1} (u(x_1)+\varepsilon \xi(x_1))^{n-1} = (n-1)\abs{B^{n-1}_1}\int_{-\pi}^\pi \dd{x_1} u(x_1)^{n-1}\xi(x_1). \label{first_var_volume} %S'exponent de u(x-1) toca ser un n-2, no?
\end{align}

We now calculate the first variation of $\mathcal{P}_K$ and show that it leads to the anisotropic nonlocal mean curvature of $\partial E$ at $x=(x_1,x')$. First we write $\mathcal{P}_K(u)$ as
\begin{align}
    \mathcal{P}_K (u) &= \int_{-\pi}^\pi \dd{x_1} \int_{-\infty}^\infty \dd{y_1} \int_{\abs{x'}<u(x_1)} \dd{x'} \int_{\abs{y'}>u(y_1)} \dd{y'} K(x_1-y_1,x'-y')\\ 
    &= \int_{-\pi}^\pi \dd{x_1} \int_{-\infty}^\infty \dd{y_1} \int_{0}^{u(x_1)} \dd{r_{x'}}   \int_{u(y_1)}^\infty \dd{r_{y'}}r_{x'}^{n-2} r_{y'}^{n-2} \phi(x_1-y_1, r_{x'}, r_{y'}),
\end{align}
where we have defined
\begin{equation}
    \phi(z,s,t):=\int_{S^{n-2}}\dd{\sigma_x}\int_{S^{n-2}} \dd{\sigma_y}  K(z,s\sigma_x - t \sigma_y), \quad \text{for }z,s,t \in \mathbb{R}, \ s,t>0.
\end{equation}
Then,
\begin{multline}
    \derivative{}{\varepsilon}\Big\vert_{\varepsilon=0} \mathcal{P}_K(u+\varepsilon \xi)=
     \int_{-\pi}^\pi \dd{x_1} \int_{-\infty}^\infty \dd{y_1} \left\{ \int_{u(y_1)}^\infty \dd{r_{y'}} r_{y'}^{n-2}u(x_1)^{n-2} \xi(x_1)\phi(x_1-y_1, u(x_1), r_{y'}) \right. \\ 
    -\left. \int_{0}^{u(x_1)} \dd{r_{x'}}   r_{x'}^{n-2} u(y_1)^{n-2} \xi(y_1) \phi(x_1-y_1, r_{x'}, u(y_1))\right\}.
\end{multline}

Changing variables and using the periodicity of $u$ and $\xi$, we see that
\begin{align}
    \derivative{}{\varepsilon}\Big\vert_{\varepsilon=0} \mathcal{P}_K(u+\varepsilon \xi)=&\int_{-\pi}^\pi \dd{x_1} \int_{-\infty}^\infty \dd{y_1} \int_{0}^{u(x_1)} \dd{r_{x'}} r_{x'}^{n-2} u(y_1)^{n-2} \xi(y_1) \phi(x_1-y_1, r_{x'}, u(y_1)) \\ 
    =& \int_{-\infty}^\infty \dd{x_1} \int_{-\pi}^\pi d y_1 \int_{0}^{u(x_1)} \dd{r_{x'}}  r_{x'}^{n-2} u(y_1)^{n-2} \xi(y_1) \phi(x_1-y_1, r_{x'}, u(y_1))\\ 
    =&\int_{-\pi}^\pi \dd{x_1} \int_{-\infty}^\infty d y_1 \int_{0}^{u(y_1)} \dd{r_{x'}}  r_{x'}^{n-2} u(x_1)^{n-2}\xi(x_1) \phi(y_1-x_1, r_{x'}, u(x_1)).
\end{align}
Therefore, using the symmetry of $K$ (assumption \eqref{H1::Symetry}), we get
\begin{align}
&\hspace{-0.75 cm}\derivative{}{\varepsilon}\Big\vert_{\varepsilon=0} \mathcal{P}_K(u+\varepsilon \xi)=\\
    &\begin{multlined}[b]
        =\int_{-\pi}^\pi \dd{x_1} \int_{-\infty}^\infty \dd{y_1} u(x_1)^{n-2} \xi(x_1) \left\{ \int_{u(y_1)}^\infty \dd{r_{y'}}  r_{y'}^{n-2} \phi(x_1-y_1, u(x_1), r_{y'}) \right. \\ 
     \left.- \int_0^{u(x_1)} \dd{r_{y'}} r^{n-2}_y\phi(x_1-y_1, r_{y'}, u(x_1))\right\}
    \end{multlined}\\
    &\begin{multlined}[b]
        =\int_{-\pi}^\pi \dd{x_1} \int_{-\infty}^\infty \dd{y_1} \xi(x_1) \left\{ \int_{\abs{x'}=u(x_1)}\dd{x'}\int_{\abs{y'}>u(y_1)} \dd{y'}  K(x_1-y_1,x'-y') \right.\\ 
    \left. \int_{\abs{x'}<u(y_1)}\dd{x'} \int_{\abs{y'}=u(x_1)} \dd{y'} K(x_1-y_1,x'-y')\right\}
    \end{multlined}\\ 
    & \begin{multlined}[b]
        =\abs{\mathbb{S}^{n-2}} \int_{-\pi}^\pi \dd{x_1} u(x_1)^{n-2}\xi(x_1)\int_{-\infty}^\infty \dd{y_1}  \left\{\int_{\abs{y'}>u(y_1)} \dd{y'} K(x_1-y_1,u(x_1)e_1'-y') \right. 
    \\ \left. - \int_{\abs{x'}<u(y_1)} \dd{x'} K(x_1-y_1,x'-u(x_1)e_1')\right\}, \label{final_expression_in_euler_lagrange}
    \end{multlined}
\end{align}
where $e_1'=(1,0,\dots,0)\in \mathbb{R}^{n-1}$. 

Now, we compute $H_K[E]$ at $x=(x_1,u(x_1)e_1')$.
\begin{align}
    H_K[E](x) &= \int_{\mathbb{R}^n} \dd{y} \left(\chi_{E^c}(y)- \chi_E(y)\right) K(x-y) \\ 
&=\begin{multlined}[t]
    \int_{-\infty}^\infty \dd{y_1} \left\{ \int_{\abs{y'}>u(y_1)} \dd{y'} - \int_{\abs{y'}<u(y_1)} \dd{y'} \right\} K(x_1-y_1,u(x_1)e_1'-y') ,\label{NLMC_arbitrary_n_expression}
\end{multlined}
\end{align}
which is precisely the expression in braces in \eqref{final_expression_in_euler_lagrange}. Hence,
\begin{equation}\label{first_variation_at_eps0}
    \derivative{}{\varepsilon}\Big\vert_{\varepsilon=0} \mathcal{P}_K(u+\varepsilon \xi)=\abs{\mathbb{S}^{n-2}} \int_{-\pi}^\pi \dd{x_1} u(x_1)^{n-2} \xi(x_1) H_K[E](x_1,u(x_1)e'_1).
\end{equation}
From this expression, \eqref{first_var_volume}, and Lagrange's multipliers rule we deduce that
\begin{equation}
    H_K[E](x_1,u(x_1)e_1') = \theta \quad \text{for all }x_1 \in U(x_1^0)
\end{equation}
for some $\theta \in \mathbb{R}$.

It remains to show that $\theta$ does not depend on $x^0$ or $U(x_1^0)$. For this, let $x^0, y^0 \in \partial E$, and $U(x_1^0)$ and $U(y_1^0)$ be neighborhoods of $x^0_1$ and $y_1^0$ where $u$ is regular. We can repeat the argument above in $U(x_1^0)\cup U(y_1^0)$ choosing $\xi$ to have compact support in $U(x_1^0)\cup U(y_1^0)$ to deduce that $H_K[E](x_1,u(x_1)e_1')$ must also be constant in $U(x_1^0)\cup U(y_1^0)$. Since it is equal to $\theta$ in $U(x_1^0)$, then it must be equal to $\theta$ in $U(x_1^0)\cup U(y_1^0)$.

\proofstep{Proof of (iv).} We denote by $B_{\omega}$ the open ball centered at $0$ with volume $\omega$. Set
\begin{equation}
    \mathcal{B}_{\omega} = \bigcup_{k\in \mathbb{Z}} \left(B_{\omega} + 2\pi k e_1\right),
\end{equation}
with $e_1=(1,0,\dots,0)\in \mathbb{R}^n$ and $\omega$ small enough that $B_\omega \subset\subset \Omega$. Let $\mathcal{C}_{\omega}$ be the straight cylinder with volume $\omega$ in $\Omega$, i.e., let $\mathcal{C}_\omega=\{(x_1,x')\in \mathbb{R}^n : \abs{x'}<R(\omega)\}$, with $R(\omega)>0$ chosen so that $\abs{\mathcal{C}_\omega \cap \Omega}=\omega$. 

From \cite[Theorem 1.1 (iv)]{CCM}, we have that there exist positive constants $c_n$ and $C_n$, depending only on $n$, such that
\begin{equation}
    \mathcal{P}_\alpha[\mathcal{B}_{\omega}] \leq \frac{C_n}{\alpha(1-\alpha)}\omega^{\frac{n-\alpha}{n} }\quad \text{and} \quad \mathcal{P}_\alpha[\mathcal{C}_\omega]\geq \frac{c_n}{\alpha (1-\alpha)}\omega^{\frac{n-1-\alpha}{n-1}}.
\end{equation}
As a consequence, by \eqref{H5::LowerBound} and \eqref{H5::UpperBound}, and since $(n-1-\alpha)/(n-1)<(n-\alpha)/n$, for $\omega$ small enough, depending only on $n$, $\alpha$, $\lambda$, and $\Lambda$, we have
\begin{equation}
    \mathcal{P}_K[\mathcal{B}_{\omega}]\leq \Lambda\frac{C_n}{\alpha(1-\alpha)}\omega^{\frac{n-\alpha}{n}}< \lambda \frac{c_n}{\alpha (1-\alpha)}\omega^{\frac{n-1-\alpha}{n-1}} \leq \mathcal{P}_K[\mathcal{C}_\omega].
\end{equation}
Hence, we conclude that $\mathcal{P}_K[\mathcal{B}_{\omega}]<\mathcal{P}_K[\mathcal{C}_\omega]$, and therefore $\mathcal{C}_\omega$ is not a minimizer.
\end{proof}

\section{A family of curves in the plane with constant anisotropic nonlocal mean curvature through the implicit function theorem} \label{sec:IFTSection}
In the following, $K: \mathbb{R}^2 \to (0,+\infty)$ is a positive measurable kernel satisfying \eqref{H1::Symetry}--- \eqref{H6::DecayLipschitzNorm}. Slightly abusing notation, we write $K(z)=K(z_1,z_2)$, with $z=(z_1,z_2)\in \mathbb{R}^2$. Since we are now in dimension $2$, assumptions \eqref{H1::Symetry} and \eqref{H2::RotInv} are equivalent to
\begin{align}
    K(-z_1,-z_2) &= K(z_1,z_2), \label{Hcyl1.1} \\
    K(z_1,-z_2) &= K(z_1,z_2) \label{Hcyl1.2},\\
    K(-z_1,z_2) &= K(z_1,z_2) \label{Hcyl1.3}
\end{align}
for every $z \in \mathbb{R}^2$. Of course, any one of the previous three equalities can be deduced from the other two, but we state them all for future reference.

Let $E=\{(z_1,z_2)\in \mathbb{R}^2 : -u(z_1)<z_2<u(z_1)\}$ for some nonnegative function $u:\mathbb{R}\to \mathbb{R}$. We denote, with a small abuse of notation,
\begin{equation}
    H_K(u)(s):=H_K[E](s,u(s)).
\end{equation}
The following lemma provides a more treatable expression for $H_K(u)$.
\begin{lemma} \label{NLMCLemma}
    Let $K: \mathbb{R}^2 \to (0,+\infty)$ satisfy \eqref{Hcyl1.2} and \eqref{Hcyl1.3}. The anisotropic nonlocal mean curvature of $E$ at the point $(s,u(s))\in \partial E$ is given by
    \begin{equation}
        \frac{1}{2} H_K(u)(s) = \int_\mathbb{R} G\left(t,u(s)-u(s-t)\right) \dd{t}
        - \int_\mathbb{R} \left\{ G\left(t,u(s)+u(s-t)\right) - G\left(t,+\infty\right) \right\} \dd{t},
    \end{equation}
    where the integrals are to be understood in the principal value sense, and
    \begin{equation}
        G(p,q) := \int_0^q K(p,\tau) \dd{\tau}.
    \end{equation}
\end{lemma}

\begin{proof}
    We have
    \begin{align}
        -H_K(u)(s) &= \int_{\mathbb{R}^2} \left\{\chi_E(s_1,s_2) - \chi_{E^c}(s_1,s_2)\right\} K(s_1-s,s_2-u(s)) \dd{s_1} \dd{s_2}\\
        &= \int_\mathbb{R} I(s,s_1) \dd{s_1},
    \end{align}
    with
    \begin{align}
        I(s,s_1) :=& \left\{ \int_{-u(s_1)}^{u(s_1)} \dd{s_2} - \int_{-\infty }^{-u(s_1)} \dd{s_2} - \int_{u(s_1)}^{+\infty } \dd{s_2} \right\} K(s_1-s,s_2-u(s))\\
        =& \left\{ 2 \int_{-u(s_1)}^{u(s_1)} \dd{s_2} - \int_\mathbb{R} \dd{s_2} \right\} K(s_1-s,s_2-u(s))\\
        =& \left\{ 2 \int_{-u(s_1)-u(s)}^{u(s_1)-u(s)} \dd{\tau} - \int_\mathbb{R} \dd{\tau} \right\} K(s_1-s,\tau),
    \end{align}
    where we have used the change of variables $\tau = s_2 - u(s)$ in the last equality. Hence, we write
    \begin{align}
        I(s,s_1) &= \begin{multlined}[t][0.8\displaywidth]
            2 \left\{ G(s_1-s,u(s_1)-u(s)) - G(s_1-s,-u(s_1)-u(s)) \right\}\\
        - G(s_1-s,+\infty) + G(s_1-s,-\infty)
        \end{multlined}\\
        &= -2 \left\{ G(s_1-s,u(s)-u(s_1)) - G(s_1-s,u(s)+u(s_1))+ G(s_1-s,+\infty) \right\}\\
        &= - 2 G(s_1-s,u(s)-u(s_1))+ 2 \left\{ G(s_1-s,u(s)+u(s_1)) - G(s_1-s,+\infty) \right\}
    \end{align}
    where in the second to last equality we have used the fact that, since $K(p, \cdot)$ is even, $G(p,\cdot)$ is odd. Finally, since $K(\cdot,q)$ is even, $G(\cdot, q)$ is even too, and thus
    \begin{equation}
        I(s,s_1) = - 2 G(s-s_1,u(s)-u(s_1)) + 2 \left\{ G(s-s_1,u(s)+u(s_1)) - G(s-s_1,+\infty) \right\}.
    \end{equation}
    The expression in the statement of the lemma follows from the change of variables $t = s - s_1$.
\end{proof}

With this formula in hand, we can easily compute the ANMC of a straight cylinder of radius $R>0$, $h_R$, which corresponds to the constant function $u_R \equiv R$:
\begin{equation} \label{h_RDef}
    h_R:=H_K(u_R) = - 2\int_\mathbb{R} \left\{ G\left(t,2R\right) - G\left(t,+\infty\right) \right\} \dd{t}.
\end{equation}
Notice that $h_R$ is indeed constant and positive.

We want to solve the equation
\begin{equation}
    \mathcal{H}(R,u):=H_K(u)-h_R = 0.
\end{equation}
Of course, for every $R>0$ we have the trivial solution $u_R\equiv R$; hence, to obtain the desired family of nontrivial solutions, we must first find a bifurcation point for $\mathcal{H}$ at which we can apply the theorem of Crandall and Rabinowitz \cite{Cr-Rab}. As we will see in Proposition \ref{linearizedOpInvCont} below, the linearization $D_u \mathcal{H}(R,u_R)$ diagonalizes in the subspace of even functions in $L^2(-\pi,\pi)$, and its eigenvectors form the symmetric Fourier basis, $e_k(s)=\cos(ks)$. We will see that there exists a unique $R^*>0$ such that $D_u \mathcal{H}(R^*,u_{R^*})$ has kernel spanned by $e_1$. The transversality condition $D^2_{u,R}\mathcal{H}(R^*,u_{R^*})e_1\neq 0$ needed to ensure that $(R^*,u_{R^*})$ is a bifurcation point for $\mathcal{H}$ will be granted by assumption \eqref{H3::RadiallyDec}---this is precisely the content of item (a) in Proposition \ref{linearizedOpInvCont}. As a consequence---ignoring for now the question of choosing the appropriate functional spaces---, the hypotheses of the bifurcation theorem are satisfied, and thus we deduce the existence of a curve of nontrivial solutions $\{ (\gamma(a)R^*, w(a)) : a\in (-\nu, \nu)\}$ branching from the trivial solution $(R^*,u_{R^*})$, with each $w(a)$ of the form ${w(a)(s) = \gamma(a)R^* + a (\cos(s)+v(a)(s))}$, for some even function $v(a)$ orthogonal to $\cos(\cdot)$ in $L^2(-\pi,\pi)$.

To make the exposition more transparent, we will not apply the bifurcation theorem directly, but, rather, we will make the construction ``by hand'', as in the original paper \cite{Cr-Rab}.

In view of the previous discussion, we consider the functions $w:\mathbb{R}\to \mathbb{R}$ given by
\begin{equation} \label{perturbationexpression}
  w(s) = \gamma R + a \varphi(s)=\gamma R + a (\cos(s)+v(s)),
\end{equation}
where $R>0$, $\gamma >0$, $a\in \mathbb{R}$, and $v:\mathbb{R}\to\mathbb{R}$ is some function orthogonal to $\cos(\cdot)$ in $L^2(-\pi,\pi)$. We want to solve the equation
\begin{equation}
    H(w)(s) -  h_{\gamma R} = 0 \quad \text{for every }s\in \mathbb{R}.
\end{equation}
Using Lemma \ref{NLMCLemma} and recalling \eqref{h_RDef}, this equation is equivalent to
\begin{equation}\label{eq_H(w)=hmuR}
    \int_\mathbb{R} \left\{ G\left(t,w(s)-w(s-t)\right) - \left\{ G\left(t,w(s)+w(s-t)\right) - G\left(t,2\gamma R\right) \right\} \right\} \dd{t} = 0
\end{equation}
for every $s\in \mathbb{R}$. As done in the classical paper by Crandall and Rabinowitz \cite{Cr-Rab}, we divide equation \eqref{eq_H(w)=hmuR} by $a$. This yields the new operator
\begin{multline} \label{mainoperatorexpression}
    \Phi(a,\gamma,\varphi)(s) := \int_\mathbb{R} \frac{1}{a} G\left(t,a (\varphi(s)-\varphi(s-t))\right) \dd{t}\\
    - \int_{\mathbb{R}} \frac{1}{a} \left\{ G\left(t,2\gamma R + a(\varphi(s)+\varphi(s-t))\right) - G\left(t,2\gamma R\right) \right\}\dd{t}.
\end{multline}

\subsection{The functional setting} \label{FunctionSpacesSubsection}
We now introduce the function spaces in which we will work. Given $\alpha\in (0,1)$, we choose $\beta$ such that
\begin{equation} \label{techCondBeta}
    0<\alpha<\beta<\min\{1,2\alpha+1/2\},
\end{equation} 
and define the spaces
\begin{equation}
  X := C^{1,\beta}_{even}(\mathbb{T}^1) = \{ \varphi : \mathbb{R} \to \mathbb{R} : \varphi \in C^{1,\beta}(\mathbb{R}) \text{ is } 2\pi\text{-periodic and even} \} , 
\end{equation}
and
\begin{equation}
  Y := C^{\beta-\alpha}_{even}(\mathbb{T}^1) = \{ \tilde\varphi : \mathbb{R} \to \mathbb{R} : \tilde\varphi \in C^{0,\beta-\alpha}(\mathbb{R}) \text{ is } 2\pi\text{-periodic and even} \}.  
\end{equation}
Since functions in these spaces are $2\pi$-periodic and even, we can endow $X$ and $Y$ with the norms of $C^{1,\beta}([0,\pi])$ and $C^{0,\beta-\alpha}([0,\pi])$, respectively. That is,
\begin{equation} \label{Xnorm}
    \|\varphi\|_X := \|\varphi\|_{L^\infty(0,\pi)} + \| \varphi'\|_{L^\infty (0,\pi)} + \sup_{0\le s < \bar s\le \pi} \frac{|\varphi'(s)-\varphi'(\bar s)|}{|s - \bar s|^\beta},
\end{equation}
and
\begin{equation}
    \|\tilde\varphi\|_Y := \|\tilde\varphi\|_{L^\infty(0,\pi)} + \sup_{0\le s < \bar s\le \pi} \frac{|\tilde\varphi(s)-\tilde\varphi(\bar s)|}{|s - \bar s|^{\beta-\alpha}}.
\end{equation}

Next, we write the expression for the operator $\Phi$ when $a=0$, which after an appropriate change of variables becomes
\begin{equation} \label{Operator_a=0}
    \Phi(0,\gamma,\varphi)(s) = \int_\mathbb{R} \left( \varphi(s)-\varphi(s-t) \right) K(t,0) \dd{t} - \int_\mathbb{R} \left( \varphi(s)+\varphi(s-t) \right) K(t,2\gamma R) \dd{t}.
\end{equation}
Notice that the resulting operator is linear in $\varphi$.\footnote{Indeed, it is precisely the operator $D_u\mathcal{H}(R,u_R)$ discussed above.} We will see that there exists a unique $R^*>0$ such that
\begin{multline}\label{RsuchthatPhi=0}
    0 = \Phi(0,1,\cos(\cdot))(s) = \int_\mathbb{R} \left( \cos(s) - \cos(s-t)  \right) K(t,0) \dd{t} \\
    - \int_\mathbb{R} \left( \cos(s) + \cos(s-t)  \right) K(t,2 R^*) \dd{t}
\end{multline}
for all $s \in \mathbb{R}$ (see Lemma \ref{lambda1=0}).

In the next section, we will study the linearized operator $D_{(\gamma,\varphi)}\Phi$ at $(a,\gamma,\varphi)=(0,1,\cos(\cdot))$. As we discussed previously, we will see that, for $R=R^*$, the function $\cos(\cdot)$ belongs to the kernel of $D_\varphi \Phi(0,1,\cos(\cdot))$. Hence, we define the the following subspaces of $X$ and $Y$.

Consider the orthogonal basis $\{1,\cos(\cdot), \cos(2\cdot),\cos(3\cdot),... \}$ of $L^2_{even}(\mathbb{T}^1)$, and the subspaces
\begin{equation}
    V_1 = \mathrm{span}\{ \cos(\cdot) \} \qquad \text{ and } \qquad V_2 = \mathrm{span} \left\{\cos(\cdot) \right\}^\perp = \overline{\mathrm{span}\left\{ 1, \cos(2\cdot),\cos(3\cdot),...\right\}}.
\end{equation}
We define
\begin{equation}
  X_1 = X\cap V_1, \quad X_2 = X\cap V_2, \quad Y_1 = Y\cap V_1, \quad Y_2 = Y\cap V_2,  
\end{equation}
together with the standard projections
\begin{equation}
    \Pi_1: Y \to Y_1 \quad \text{ and } \quad \Pi_2: Y \to Y_2.
\end{equation}

Recall that we are looking for functions $\varphi$ of the form
\begin{equation}
    \varphi(s) = \cos(s) + v(s), \quad v \in X_2.
\end{equation}
For convenience, we will sometimes write our operator acting on functions $v \in X_2$ as
\begin{equation}
  \bar\Phi : A \subset \mathbb{R} \times \mathbb{R} \times X_2 \to Y, \qquad \bar \Phi (a,\gamma,v) := \Phi (a,\gamma, \cos(\cdot) + v),  
\end{equation}
with $A$ some open neighborhood of $(0,1,0)$.

All the necessary conditions needed to apply the implicit function theorem are stated in the following proposition, whose proof will be given in the next two subsections.
\begin{proposition} \label{ImplFunctHypProp}
    Let $K:\mathbb{R}^2\to (0,+\infty)$ be a positive measurable kernel satisfying \eqref{H1::Symetry}---\eqref{H6::DecayLipschitzNorm}. There exist $R^*>0$ and $\nu>0$, depending only on $K$, for which the operator
    \begin{equation}
        \bar \Phi  : (-\nu,\nu) \times (1/2,3/2) \times B_1(0) \subset \mathbb{R} \times \mathbb{R} \times X_2 \to Y
    \end{equation}
    is of class $C^1$. Moreover, there exists $\varepsilon_0>0$ such that if $K$ satisfies, in addition, either \eqref{H7::CloseToHom_1} or \eqref{H7::CloseToHom_2} for some $\varepsilon\in (0,\varepsilon_0)$, then the linear operator
    \begin{equation}
        (D_\gamma \bar \Phi, D_v \bar\Phi)(0,1,0): \mathbb{R}\times X_2 \to Y
    \end{equation}
     is continuous and invertible.
\end{proposition}

With this, we may now prove our main result.
\begin{proof}[Proof of Theorem \ref{mainresultperturbative}]
    By \eqref{RsuchthatPhi=0} we have $\bar \Phi (0,1,0) = 0$. Proposition \ref{ImplFunctHypProp} gives the necessary conditions to apply the implicit function theorem in order to obtain the family of functions $w(a)$ as in the statement of the theorem. All properties are derived directly from the setting described in this section. 

    The only thing left to prove is that each $w(a)$ has prime period equal to $2\pi$. For this, we write $v(a)(s)=c_0+\sum_{k=2}^\infty c_k \cos(ks)$. Let $T$ be the prime period of $w(a)$. Then,
    \begin{multline}
        \cos(s)+v(a)(s) = \cos(s+T)+v(a)(s+T)  \\ =\cos(s)\cos(T)-\sin(s)\sin(T)+a_0+\sum_{k=2}^\infty c_k \left\{\cos(ks)\cos(kT)-\sin(ks)\sin(kT)\right\}.
    \end{multline}
    Multiplying through by $\cos(s)$ and integrating in $(-\pi,\pi)$ we deduce, by orthogonality, that $\cos(T)=1$. Hence, since $w(a)$ is $2\pi$-periodic, we conclude that $T=2\pi$.
\end{proof}

\begin{proof}[Proof of Corollary \ref{corollary_hom_K}]
    Suppose that $K(\lambda z)=\lambda^{-2-\alpha}K(z)$ for every $\lambda>0$ and every $z\in \mathbb{R}^2$. Let $w(a)$ solve $H(w(a))=h_{\gamma(a) R}$, i.e., recalling \eqref{eq_H(w)=hmuR},
    \begin{equation}
    \int_\mathbb{R} \left\{ G\left(t,w(a)(s)-w(a)(s-t)\right) - \left\{ G\left(t,w(a)(s)+w(a)(s-t)\right) - G\left(t,2\gamma(a) R^*\right) \right\} \right\} \dd{t} = 0.
\end{equation}
    It is easy to check, using the equality
    \begin{equation}
        G(\gamma p, \gamma q) = \int_0^{\gamma q}K(\gamma p, \tau) \dd{\tau} = \gamma^{-2-\alpha+1} \int_0^q K(p,\sigma)d\sigma=\gamma^{-1-\alpha}G(p,q),
    \end{equation}
    that $\tilde{w}(a)(\cdot)=\gamma(a)^{-1}w(a)(\gamma(a)\cdot)$ solves
    \begin{equation}
    \int_\mathbb{R} \left\{ G\left(t,\tilde{w}(a)(s)-\tilde{w}(a)(s-t)\right) - \left\{ G\left(t,\tilde{w}(a)(s)+\tilde{w}(a)(s-t)\right) - G\left(t,2R^*\right) \right\} \right\} \dd{t} = 0,
\end{equation}
    i.e., $H(\tilde{w}(a))=h_R$.

    To show that $\tilde{w}(a)\not\equiv \tilde{w}(a')$ for $a\neq a'$, let us assume that $\tilde{w}(a)\equiv \tilde{w}(a')$. Then, their prime periods must coincide and thus $\gamma(a)=\gamma(a')=\bar{\gamma}$. As a consequence,
    \begin{equation}
        a\left(\cos(\bar{\gamma}s)+v(a)(\bar{\gamma}s)\right) = a'\left(\cos(\bar{\gamma}s)+v(a')(\bar{\gamma}s)\right)
    \end{equation}
    and, by orthogonality of $v(a)$ and $v(a')$ with $\cos(\cdot)$, we deduce that $a=a'$.
\end{proof}

The following two subsections are devoted to the proof of Proposition \ref{ImplFunctHypProp} above. First, in Subsection \ref{InvertibilitySubsection}, we prove the invertibility of the operator $\Phi$. Finally, in Subsection \ref{Regularitysubsection}, we prove that it is $C^1$.

\subsection{The linearized operator acting on even periodic functions} \label{InvertibilitySubsection}
We devote this subsection to the proof of the following proposition, of which the second half of Proposition \ref{ImplFunctHypProp} is an immediate consequence.
\begin{proposition}\label{linearizedOpInvCont}
    Let $K:\mathbb{R}^2\to (0,+\infty)$ be a positive kernel satisfying \eqref{H1::Symetry}---\eqref{H6::DecayLipschitzNorm}. 
    There exist $\varepsilon_0>0$ and a unique $R^*>0$ such that, if $K$ satisfies, in addition, either \eqref{H7::CloseToHom_1} or $\eqref{H7::CloseToHom_2}$ for some $\varepsilon\in (0,\varepsilon_0)$, then, the following hold:
    \begin{enumerate}[(a)]
        \item There is a constant $\kappa >0$ such that
    \begin{equation}
        D_\gamma \bar\Phi (0,1,0)= \kappa \cos(\cdot).
    \end{equation}

    \item For every $\psi \in X$,
    \begin{equation}
        L \psi := D_v \bar\Phi(0,1,0)\psi = \Gamma \psi -\left(\int_\mathbb{R} P_{R^*} \right) \psi - P_{R^*} \ast \psi,
    \end{equation}
    with
    \begin{equation}
        \Gamma \psi(s) := \int_\mathbb{R} (\psi(s)-\psi(s-t)) K(t,0) \dd{t}, \quad \text{and} \quad
        P_{R^*}(t):=K(t,2R^*). \label{GammaOperDef}
    \end{equation}
    Moreover, the functions $e_k(s) =  \cos(ks)$, $k=0,1,2,...$, are all eigenfunctions of $L$, with eigenvalues satisfying
    \begin{equation}
        Le_k = \mu_k e_k, \quad \mu_0 < 0, \quad \mu_1 = 0, \quad 0<\mu_k \quad \text{for} \quad k\ge 2,
\label{EigenvaluesProperties}
    \end{equation}
    and
    \begin{equation}
        0<\liminf_{k\to\infty}\frac{\mu_k}{k^{1+\alpha}}\leq \limsup_{k\to\infty}\frac{\mu_k}{k^{1+\alpha}} < + \infty.    \label{AssymptBehaviour}
        \end{equation}
\end{enumerate}

As a consequence, the linear operator
    \begin{equation}
        (D_\gamma \bar \Phi , D_v \bar \Phi) (0,1,0) : \mathbb{R} \times X_2 \to Y
    \end{equation}
    is continuous and invertible.
\end{proposition}

Let $R>0$ be fixed, to be determined later. We prove each claim separately.

\proofstep{Step 1. Proof of (a).} From \eqref{Operator_a=0}, we have
\begin{align} 
    D_\gamma \bar \Phi(0,1,0)(s) &= - \int_\mathbb{R} 2R(\cos(s)+\cos(s-t)) K_{z_2} (t,2 R)\dd{t}\\
    &= - \cos(s) \int_\mathbb{R} 2R(1+\cos(t))  K_{z_2} (t,2 R)\dd{t}=:\kappa \cos(s),\label{DlambdaPhi}
\end{align}
where we have used that $\cos(s-t) = \cos(s) \cos(t) + \sin(s) \sin(t)$, that $\sin(t)$ is an odd function, and that $K_{z_2} (t,2 R)$ is an even function of $t$ due to \eqref{Hcyl1.3}. Note that $\kappa$ is a positive number, since $1+\cos(t) \ge 0$, $K_{z_2} (t,2\gamma R) \le 0$ due to \eqref{H3::RadiallyDec}, and \eqref{H5::LowerBound} and \eqref{H5::UpperBound} ensure that $K_{z_2} (t,2 \gamma R) \not \equiv 0$.

From \eqref{DlambdaPhi} we deduce that
\begin{equation}
    D_\gamma \left( \Pi_1 \bar \Phi \right) (0,1,0)(s) = - \kappa \cos(s)
\quad \text{and}\quad
    D_\gamma \left( \Pi_2 \bar \Phi \right) (0,1,0)(s) \equiv 0.
\end{equation}

\proofstep{Step 2. Proof of (b): nondegeneracy of the operator $L$.}
From \eqref{Operator_a=0}, we have
\begin{align} 
    L \psi (s) &= D_v \bar \Phi (0,1,0) \psi(s)\\
    &= \int_\mathbb{R} \left( \psi(s) - \psi(s-t) \right) K(t,0)\dd{t} - \int_\mathbb{R} \left( \psi(s) + \psi(s-t) \right) K(t,2R) \dd{t} \label{LwAux}\\
    &= \left\{\Gamma \psi - \left(\int_\mathbb{R} P_R \right) \psi - P_R \ast \psi \right\} (s). \label{Lw_gamma_PR}
\end{align}

Since the operator $L$ is linear and translation invariant, it is not difficult to see that the Fourier basis ---restricted to the space of even functions--- is a basis of eigenvectors of $L$. The precise statement of this fact is contained in the following lemma, which is a straightforward generalization of \cite[Lemma 5.2]{CNLMCDelCyl}.
\begin{lemma} \label{ConvLemma}
    Let $e_k(s)=\cos(ks)$, with $k =0,1,2...$
    \begin{enumerate}[(i)]
        \item If $P \in L^1(\mathbb{R})$ is an even function, then
        \begin{equation}
            (P \ast e_k)(s) = \left( \int_\mathbb{R} \cos(kt)P(t) \dd{t} \right) e_k(s)
        \end{equation}
        for every $s \in \mathbb{R}$. Hence, the functions $e_k$ are eigenfunctions of the operator $f\mapsto P\ast f$.
        \item The functions $e_k$ are also eigenfunctions of the operator $\Gamma$ defined in \eqref{GammaOperDef}, and
        \begin{equation}
            \Gamma e_k(s) = \left( \int_\mathbb{R} (1-\cos(kt)) K(t,0) \dd{t}\right) e_k(s),
        \end{equation}
        where the integral is taken in the principal value sense.
    \end{enumerate}
\end{lemma}

Continuing from \eqref{Lw_gamma_PR}, using Lemma \ref{ConvLemma} we deduce that
\begin{equation}
    Le_k = \left\{ \int_\mathbb{R} (1-\cos(kt)) K(t,0) \dd{t} - \int_\mathbb{R} (1 +\cos(kt)) K(t,2R) \dd{t} \right\} e_k =: \mu_k(R) e_k,
\end{equation}
for $k=0,1,2,...$, with
\begin{equation}\label{expr_muk}
     \mu_k(R) = \int_\mathbb{R} (1-\cos(kt)) K(t,0) \dd{t} - \int_\mathbb{R} (1 +\cos(kt)) K(t,2R) \dd{t}.
\end{equation}
Notice that $\mu_0(R) = -2 \int_\mathbb{R} K(t,2R) \dd{t} <0$ for every $R>0$.

We now prove that there exists a unique $R^*>0$ for which $\mu_1(R^*)=0$.
\begin{lemma}\label{lambda1=0}
    There exists a unique $R^*>0$, depending only $K$, such that
    \begin{equation}
        \mu_1(R^*) = \int_\mathbb{R} (1-\cos(t)) K(t,0) \dd{t} - \int_\mathbb{R} (1+\cos(t)) K(t,2R^*) \dd{t} = 0.
    \end{equation}
\end{lemma}
\begin{proof}
    We first differentiate $\mu_1$ with respect to $R$ and notice that, due to \eqref{H3::RadiallyDec}, \eqref{H5::LowerBound}, and \eqref{H5::UpperBound},
    \begin{equation}
      \mu_1'(R) = -2 \int_\mathbb{R} (1+\cos(t)) K_{z_2}(t,2R) \dd{t} > 0,  
    \end{equation}
    and thus $\mu_1$ is increasing in $R$. By the monotone convergence theorem, \eqref{H5::LowerBound}, and \eqref{H5::UpperBound},
    \begin{equation}
        \mu_1(R) \uparrow \int_\mathbb{R} (1-\cos(t)) K(t,0) \dd{t} > 0\text{ as } R \uparrow \infty,
    \end{equation}
    and
    \begin{equation}
      \mu_1(R) \downarrow -2 \int_\mathbb{R} \cos(t) K(t,0) \dd{t} = - \infty \text{ as } R \downarrow 0.  
    \end{equation}
    Therefore, there exists a unique $R^*>0$ such that $\mu_1(R^*)=0$.
\end{proof}

Notice that the lemma above gives the value of $R>0$ for which \eqref{RsuchthatPhi=0} is satisfied, since
\begin{align}
    \Phi(0,1,\cos(\cdot))(s) &=\int_\mathbb{R} \left( \cos(s) - \cos(s-t)  \right) K(t,0) \dd{t}
    - \int_\mathbb{R} \left( \cos(s) + \cos(s-t)  \right) K(t,2 R^*) \dd{t} \\
    &= L e_1=\mu_1(R^*) e_1 = 0.
\end{align}

Next, we prove that, if the kernel $K$ is sufficiently close to homogeneous in the sense of \eqref{H7::CloseToHom_1} or \eqref{H7::CloseToHom_2}, then, the eigenvalues $\mu_k$ are positive for $k\ge 2$.
%In fact, for one of the conditions we get a stronger property, namely that the $\mu_k$ are stricly increasing, and thus we have $\mu_0 < 0=\mu_1 < \mu_2 < \mu_3 < \cdots$. 

\proofstep{Case 1. Kernel close to homogeneous in the sense of \eqref{H7::CloseToHom_1}.}
Assume that $K$ satisfies \eqref{H7::CloseToHom_1}; that is, assume that
\begin{equation}
    \abs{K(z)-K_0(z)}\leq \varepsilon K_0(z)
\end{equation}
for some $\varepsilon>0$ to be determined lated, with $K_0$ some $-(2+\alpha)$-homogeneous nonnegative kernel satisfying \eqref{H3::RadiallyDec} and \eqref{H5::UpperBound}. Changing variables in \eqref{expr_muk} we obtain
\begin{align}
    \frac{\mu_k}{k^{1+ \alpha}} &= \frac{1}{k^{1+ \alpha}}\int_{\mathbb{R}} \left\{ (1-\cos(s))K\left(\frac{s}{k},0\right)-(1+\cos(s))K\left(\frac{s}{k},2R^*\right)\right\}\frac{\dd{s}}{k}\\[6 pt]
    &\geq \frac{1}{k^{2+ \alpha}}\int_{\mathbb{R}} \left\{ (1-\cos(s))(1- \varepsilon) K_0\left(\frac{s}{k},0\right) -(1+\cos(s))(1+ \varepsilon)K_0\left(\frac{s}{k},2R^*\right)\right\}\dd{s}\\[6 pt]
    &=\int_{\mathbb{R}} \left\{ (1-\cos(s))(1- \varepsilon) K_0\left(s,0\right) -(1+\cos(s))(1+ \varepsilon)K_0\left(s,2R^*k\right)\right\}\dd{s}=:\eta_k.
    %&\geq C\int_{\mathbb{R}} \left\{ (1-\cos(s))\frac{1- \varepsilon}{s^{2+ \alpha}}-(1+\cos(s))\frac{1}{\left(s^2+(2Rk)^{2}\right)^{\frac{2+ \alpha}{2}}}\right\}d s=:C\eta_k.
\end{align}

To prove that $\mu_k>0$ for $k\geq 2$, it is enough to show that $\eta_k>0$ for $k\geq 2$. For this, notice first that, since $K_0(s,\cdot)$ is decreasing in $(0,+\infty)$ for every $s\in \mathbb{R}$, the sequence $\eta_k$ is increasing in $k$, independently of $\varepsilon>0$. As a consequence, there is some $c>0$, which does not depend on $\varepsilon$, such that
\begin{equation}\label{eta2_bigger_than_eta1}
    \eta_2-\eta_{1} \geq c >0.
\end{equation}

Now, we show that $\eta_1$ may be negative but small. Indeed, by the choice $R=R^*$ we have that
\begin{equation}
    \mu_1=\int_{\mathbb{R}} \{(1-\cos(s)) K(s,0) - (1+\cos(s))K(s,2R^*)\} \dd{s} =0,
\end{equation}
therefore,
\begin{align}
\abs{\eta_1}&=\abs{\eta_1-\mu_1}\\ &\leq 
\begin{multlined}[t]
\abs{\int_{\mathbb{R}} \{(1-\cos(s)) \left(K_0(s,0)-K(s,0) \right)  - (1+\cos(s)) \left(K_0(s,2R^*)- K(s,2R^*) \right)\} \dd{s} }\\
    + \varepsilon \abs{\int_{\mathbb{R}} (1- \cos(s)) K_0(s,0) - (1+\cos(s)) K_0(s,2R^*) \dd{s}}
\end{multlined}\\ 
    &\leq \int_{\mathbb{R}} \{(1-\cos(s)) \abs{K_0(s,0)-K(s,0)} - (1+\cos(s))\abs{K_0(s,2R^*)-K(s,2R^*)}\} \dd{s}+O( \varepsilon)\\
    &\leq \varepsilon \int_{\mathbb{R}}\{(1-\cos(s)) K_0(s,0) +(1+\cos(s))K_0(s,2R^*)\} \dd{s} + O(\varepsilon)=O(\varepsilon).
\end{align}
Hence, by \eqref{eta2_bigger_than_eta1} and the discussion above it, for $\varepsilon$ sufficiently small we deduce that the sequence $\eta_k$ is positive for all $k\geq 2$ and, as a consequence,  $\mu_k>0$ for all $k\geq 2$.

\proofstep{Case 2. Kernel close to homogeneous in the sense of \eqref{H7::CloseToHom_2}.} Assume that $K$ satisfies \eqref{H7::CloseToHom_2}; that is, assume that
\begin{equation} \label{DerivCond}
    (2 + \alpha + \varepsilon)K(z)\geq -z\cdot \nabla K(z) \geq (2 + \alpha - \varepsilon)K(z),
\end{equation}
or, equivalently,
\begin{equation}\label{inequality_close_to_hom_2}
    (1+ \alpha + \varepsilon)K(z)+z_2K_{z_2}(z)\geq -\derivative{}{z_1}\left(z_1K(z)\right)\geq (1 + \alpha - \varepsilon)K(z)+z_2K_{z_2}(z).
\end{equation}
A calculation shows that
\begin{align}
\frac{1}{2}\derivative{\mu_k}{k} =&\int_0^\infty \left\{ \derivative{k} (1-\cos(kt)) K(t,0) - \derivative{k} (1+\cos(kt))  K(t,2R^*) \right\} \dd{t} \\[6 pt]
 =&\int_0^\infty \left\{ \derivative{t} (1-\cos(kt)) \frac{t}{k} K(t,0) - \derivative{t} (1+\cos(kt)) \frac{t}{k} K(t,2R^*) \right\} \dd{t}  \\[6 pt] 
=&-\frac{1}{k} \int_0^\infty\left\{ (1-\cos(kt))\derivative{t}\left(tK(t,0)\right) - (1+\cos(kt))\derivative{t}\left(tK(t,2R^*)\right) \right\} \dd{t}. \label{second_expr_lambda}
\end{align}
Therefore, by \eqref{inequality_close_to_hom_2},
\begin{align}
\frac{k}{2}\derivative{\mu_k}{k} &\geq\begin{multlined}[t]
    \int_0^\infty \left(1-\cos(kt)\right)(1 + \alpha - \varepsilon)K(t,0) \dd{t} \\  - \int_0^\infty (1+\cos(kt))\left\{(1+\alpha+ \varepsilon )K(t,2R^*)+2R^*K_{z_2}(t,2R^*)\right\} \dd{t}
\end{multlined}\\
&= (1+\alpha-\varepsilon) \mu_k -2\int_0^\infty (1+\cos(kt))\left\{ \varepsilon K(t,2R^*)+RK_{z_2}(t,2R^*) \right\}\dd{t}.\label{eq::lower_bound_derivative_muk}
\end{align}

Now, since $K_{z_2}(t,\cdot)\leq 0$ in $(0,+\infty)$ for every $t\in \mathbb{R}$ and $K_{z_2}(t,\cdot)\not\equiv 0$, for every $M\geq 1$ we can choose $\varepsilon>$ small enough that 
\begin{equation}\label{eq::expression_in_lower_bd_muk}
-2\int_0^\infty (1+\cos(kt))\left\{ \varepsilon K(t,2R^*)+RK_{z_2}(t,2R^*) \right\}\dd{t}  >0
\end{equation}
for every $k\in [1,M]$. Since $\mu_1=0$, from \eqref{eq::lower_bound_derivative_muk} and Gronwall's lemma we deduce that $\{\mu_k\}$ is an increasing sequence for $k\in \{1,\dots,M\}$. On the other hand, since \eqref{eq::expression_in_lower_bd_muk} is bounded as a function of $k$, as a consequence of \eqref{eq::lower_bound_derivative_muk} and \eqref{asym_lowe_bound_eigenvalues} below we deduce that the sequence $\{\mu_k\}$ is, in fact, increasing for all $k\geq 1$.

\proofstep{Step 3. Proof of (b): asymptotic behavior of the eigenvalues.}
Next we prove \eqref{AssymptBehaviour}. Notice first that, using \eqref{H5::LowerBound}, \eqref{H5::UpperBound}, and the dominated convergence theorem,
\begin{align}
\frac{\mu_k }{k^{1+\alpha}} &= \int_{\mathbb{R}} \left\{(1-\cos(kt))K(t,0)-(1+\cos(kt))K(t,2R^*)\right\} \frac{\dd{t}}{k^{1+\alpha}}\\ 
&\leq \int_{\mathbb{R}} \left\{\Lambda \frac{1-\cos(kt)}{|kt|^{2+\alpha}}-\lambda \frac{1+\cos(kt)}{((kt)^2+(2kR)^2)^{\frac{2+\alpha}{2}}}\right\}  \dd{(kt)}\\ 
& = \int_{\mathbb{R}} \left\{ \Lambda \frac{1-\cos(\bar{t})}{|\bar{t}|^{2+\alpha}}- \lambda \frac{1+\cos(\bar{t})}{(\bar{t}^2+(2kR)^2)^{\frac{2+\alpha}{2}}}\right\} \dd{\bar{t}}\xrightarrow[k\to\infty]{} \Lambda  \int_{\mathbb{R}} \frac{1-\cos(\bar{t})}{|\bar{t}|^{2+\alpha}}\dd{\bar{t}}<+\infty,
\end{align}
whence
\begin{equation}
    \limsup_{k\to\infty}\frac{\mu_k}{k^{1+\alpha}} < + \infty.
\end{equation}
Analogously we can show that
\begin{equation}\label{asym_lowe_bound_eigenvalues}
        \liminf_{k\to\infty}\frac{\mu_k}{k^{1+\alpha}} \geq \lambda \int_{\mathbb{R}}\frac{1-\cos(\bar{t})}{|\bar{t}|^{2+\alpha}} \dd{\bar{t}} >0.
\end{equation}
This completes the proof of $(b)$ in Proposition \ref{linearizedOpInvCont}.

\proofstep{Step 4. Continuity and invertibility of the operator $L$.}
 We first claim that the operator
\begin{equation} \label{L|_V_2}
    L |_{V_2} = \Pi_2 L |_{V_2} : H^{1+\alpha}_{even}(\mathbb{T}^1) \cap V_2 \to L^2_{even}(\mathbb{T}^1) \cap V_2
\end{equation}
is continuous and invertible, where, we recall, $V_2=\mathrm{span}\{\cos(\cdot)\}^\perp$, and $H^{1+\alpha}_{even}(\mathbb{T}^1)$ and $L^2_{even}(\mathbb{T}^1)$ are the spaces of $2\pi$-periodic even functions whose restrictions to $(-\pi,\pi)$ belong to $H^{1+\alpha}(-\pi, \pi)$ and $L^2(-\pi, \pi)$, respectively. Indeed, for every $f\in L^2_{even}(\mathbb{T}^1)\cap V_2$ we have
\begin{equation}
    f(s) = a^f_0+\sum_{k=2}^\infty a^f_k \cos(ks).
\end{equation}
Setting
\begin{equation}
    \tilde{w}(s)=\frac{a^f_0}{\mu_0}+\sum_{k=2}^\infty \frac{a^f_k}{\mu_k}\cos(ks),
\end{equation}
from \eqref{AssymptBehaviour} it follows that $\tilde{w}\in H^{1+\alpha}_{even}(\mathbb{T}^1)\cap V_2$, since
\begin{equation}
\norm{\tilde{w}}_{H^{1+\alpha}}^2 =
    \sum_{k=2}^\infty (1+k^2)^{1+\alpha}\bigg\vert \frac{a^f_k}{\mu_k}\bigg\vert^2 \leq  C \sum_{k=2}^\infty \bigg\vert\frac{k^{1+\alpha}}{\mu_k}\bigg\vert^2\smabs{a^f_k}^2 \leq C \sum_{k=2}^\infty \smabs{a^f_k}^2 < + \infty.
\end{equation}
Clearly, $L\vert_{V_2}\tilde{w}=f$ by construction. Moreover, using \eqref{AssymptBehaviour} again, for every $w\in H^{1+\alpha}_{even}(\mathbb{T}^1)\cap V_2$,
\begin{equation}
    \norm{Lw}_{L^2(\mathbb{R})}^2 = \abs{\mu_0 a^w_0}^2 + \sum_{k=2}^\infty \abs{\mu_k}^2\abs{a^w_k}^2 \leq \abs{\mu_0 a^w_0}^2 + C\sum_{k=2}^\infty \abs{k^{1+\alpha}}^2\abs{a^w_k}^2 \leq C \norm{w}^2_{H^{1+\alpha}},
\end{equation}
which shows that the operator in \eqref{L|_V_2} is bounded.

Next, we claim that $L$ maps $X_2$ onto $Y_2$, and, therefore, that it is invertible as an operator between $X_2$ and $Y_2$. In order to show this, let $f\in Y_2 \subset L^2_{even}(\mathbb{T}^1) \cap V_2$. Since the operator in \eqref{L|_V_2} is invertible, there exists a unique $w \in H^{1+\alpha}_{even}(\mathbb{T}^1)\cap V_2$ such that $Lw=f$. We must show that $w\in X_2$. Recall from \eqref{Lw_gamma_PR} that
\begin{equation}
  Lw = \Gamma w - c w - P_R \ast w,  
\end{equation}
for a certain $c>0$, hence, $Lw=f$ is equivalent to
\begin{equation} \label{Lweqf}
    \Gamma w = c w + P_R \ast w + f.
\end{equation}
First, since $w\in H^{1+\alpha}_{even}(\mathbb{T}^1)$ and $P_R\in  L^1 (\mathbb{R})$, we have that $P_R\ast w \in H^{1+\alpha}_{even}(\mathbb{T}^1)$. Indeed, the periodicity and evenness follow from the properties of convolution, and using the Fourier transform definition of the $H^{1+\alpha}$ norm we obtain
\begin{align}
    \norm{P_R\ast w}^2_{H^{1+\alpha}(I)}&=\int_{\mathbb{R}} (1+|\xi|^{2(1+\alpha)}) |\mathscr{F}(P_R\ast w)|^2 \dd{\xi} \\ &= \int_{I} (1+|\xi|^{2(1+\alpha)}) |\mathscr{F}(P_R)|^2|\mathscr{F}(w)|^2 \dd{\xi}\\
    &= \| \mathscr{F}(P_R) \|_{L^\infty(\mathbb{R})}^2 \norm{w}_{H^{1+\alpha}(I)}^2
\end{align}
for every compact interval $I\subset \mathbb{R}$. Moreover, since $(2(1+\alpha)-1)/2 = 1/2 + \alpha > \beta - \alpha$ (recall \eqref{techCondBeta}) by Morrey's embedding we have
\begin{equation}
    H^{1+\alpha}_{even}(\mathbb{T}^1)\subset Y = C^{\beta-\alpha}_{even}(\mathbb{T}^1).
\end{equation}
Hence, the right-hand side of \eqref{Lweqf} belongs to $Y=C^{\beta-\alpha}_{even}(\mathbb{T}^1)$. By Hölder regularity estimates for integro-differential operators (see \cite[Proposition 2.4.4]{RosFRealBook}), we get $w\in X=C^{1,\beta}_{even}(\mathbb{T}^1)$.

The only thing left to complete the proof of Proposition \ref{linearizedOpInvCont} is to prove that $L$ is continuous as an operator from $X_2$ to $Y$. In the next subsection we will prove that $\bar\Phi$ is $C^1$ from $\mathbb{R} \times \mathbb{R} \times X_2$ to $Y$, and as a consequence, since  $L|_{X_2} = D_v\bar\Phi(0,1,0)$, that $L$ maps $X_2$ continuously to $Y$ (and thus also onto $Y_2$).

\subsection{Differentiability of the nonlinear operator acting on even periodic functions} \label{Regularitysubsection}

In this section we prove the $C^1$ character of $\bar\Phi$ stated in Proposition \ref{ImplFunctHypProp}. It suffices to show that
\begin{equation}
    \Phi :(-\nu,\nu) \times (1/2,3/2) \times B_{10}(0)\subset \mathbb{R} \times \mathbb{R} \times X \to Y
\end{equation}
is of class $C^1$, since, in our setting, $\varphi = \cos(\cdot)+v$, with $v \in B_1(0)\subset X$, and
\begin{equation}
    \|\cos(\cdot)\|_X \le 1+1+\pi < 9.
\end{equation}
We split the operator $\Phi$ in two by setting
\begin{equation}
  \Phi(a,\gamma,\varphi)(s) = \Phi_1(a,\varphi) - \Phi_2(a,\gamma,\varphi),  
\end{equation}
with
\begin{equation}
    \Phi_1(a,\varphi) = \int_\mathbb{R} \frac{1}{a} G\left(t,a \delta_- \varphi\right) \dd{t}
\end{equation}
and
\begin{equation}
    \Phi_2(a,\gamma,\varphi) = \int_{\mathbb{R}} \frac{1}{a} \left\{ G\left(t,2\gamma R + a\delta_0 \varphi\right) - G\left(t,2\gamma R\right) \right\}\dd{t},
\end{equation}
where we have introduced the notations
\begin{gather}
    \delta_- \varphi(s,t) := \varphi(s) - \varphi(s-t),\\
    \delta_+ \varphi(s,t) := \varphi(s) - \varphi(s+t), \label{delta_+}\\ \intertext{and}
    \delta_0 \varphi(s,t) := \varphi(s) + \varphi(s-t) 
\end{gather}
($\delta_+$ will be used later).%\footnote{We point out that in spite of using the same notation as \cite{CNLMCDelCyl} for the operators $\delta_-,\delta_+$, the actual expressions are different (in \cite{CNLMCDelCyl} they have a $|t|$ factor dividing them).}

To express $\Phi_1$ in a nicer manner, observe that
\begin{equation}
    G_1(a,p,q) := \frac{1}{a} G(p,aq) = \frac{1}{a} \int_0^{aq} K(p,\tau) \dd{\tau} = \int_0^q K(p,a\rho) \dd{\rho}
\end{equation}
is a $C^1$ function of $q$ whenever $p\neq 0$, and that $G_1(0,p,q) = q K(p,0)$. Thus, we have
\begin{equation} \label{Phi1}
    \Phi_1(a,\varphi) = \int_\mathbb{R} G_1(a,t,\delta_- \varphi) \dd{t}.
\end{equation}

Let us also obtain a nicer expression for $\Phi_2$. To do so, consider the change of coordinates $\tau = 2\gamma R + a \delta_0 \varphi \bar \tau$, from which
\begin{align}
    \frac{1}{a} \left\{ G\left(t,2\gamma R + a\delta_0 \varphi\right) - G\left(t,2\gamma R\right) \right\} &= \frac{1}{a} \int_{2\gamma R}^{2\gamma R + a \delta_0 \varphi} K(t,\tau) \dd{\tau}= \delta_0 \varphi \int_0^1 K(t, 2\gamma R + a \delta_0 \varphi \bar \tau) \dd{\bar\tau}.
\end{align}
Thus,
\begin{equation} \label{Phi2}
    \Phi_2(a,\gamma,\varphi) = \int_\mathbb{R} \delta_0 \varphi G_2(t,a,\gamma, \delta_0 \varphi) \dd{t},
\end{equation}
where
\begin{equation}\label{G2}
    G_2(t,a,\gamma,r) = \int_0^1 K(t, 2\gamma R + a r \bar \tau) \dd{\bar\tau}.
\end{equation}

Since $\Phi=\Phi_1 - \Phi_2$, it is enough to prove that $\Phi_1$ and $\Phi_2$ are both $C^1$.

\begin{lemma} \label{Phi1C1}
    Assume that $K:\mathbb{R}^2 \to (0,+\infty)$ satisfies \eqref{H1::Symetry}---\eqref{H6::DecayLipschitzNorm}. Then, the operator
    \begin{equation}
     \Phi_1 : \mathbb{R}\times X \to Y   
    \end{equation}
    is of class $C^1$.
\end{lemma}

\begin{proof}
    Applying the change $t \to -t$ in \eqref{Phi1}, using that $\delta_-\varphi(s,-t)=\delta_+\varphi(s,t)$, and \eqref{Hcyl1.1}, we obtain
    \begin{equation}
        \Phi_1(a,\varphi) = \frac{1}{2} \int_\mathbb{R} \left\{ G_1(a, t, \delta_-\varphi) + G_1(a, t, \delta_+\varphi) \right\} \dd{t}.
    \end{equation}
    Moreover, by \eqref{Hcyl1.2},
    \begin{align}
        G_1(a, t, q) + G_1(a, t, p) &= \left\{\int_0^q + \int_0^p\right\} K(t,a\rho)\dd{\rho}\\
        &=  \left\{\int_0^q  + \int_{-p}^0\right\}  K(t,a\rho)\dd{\rho}\\
        &= \int_{-p}^q K(t,a\rho)\dd{\rho}.
    \end{align}
    The change of variables $\rho= -p+(q+p)\tau$ yields
    \begin{equation}
      G_1(a, t, q) + G_1(a, t, p) = (q+p) G_3(a,t,q,p),  
    \end{equation}
    with
    \begin{equation}
        G_3(a,t,q,p) := \int_0^1 K(t, a(-p+(q+p)\tau)) \dd{\tau},
    \end{equation}
    whence
    \begin{equation} \label{Phi1wG3}
        \Phi_1(a,\varphi) = \frac{1}{2} \int_\mathbb{R} (\delta_-\varphi + \delta_+\varphi) G_3(a,t,\delta_-\varphi,\delta_+\varphi) \dd{t}.
    \end{equation}

    With this expression in mind, we collect several inequalities for functions in $X$ ---for which one should recall the norm \eqref{Xnorm} in $X$. For all $\varphi \in X$ and $s,\bar s,t \in \mathbb{R}$ we have
    \begin{align}
        |\delta_-\varphi(s,t)| + |\delta_+\varphi(s,t)| &\le 2\|\varphi \|_X |t|, \label{Xineq1} \\
        |(\delta_-\varphi + \delta_+\varphi)(s,t)| &\le 2 \|\varphi\|_X |t|^{1+\beta}, \label{Xineq2}\\
        |\delta_\pm\varphi(s,t) - \delta_\pm\varphi(\bar s,t)|&\le \|\varphi\|_X |s-\bar s| |t|^{\beta}. \label{Xineq3}
    \end{align}
    The proofs of these inequalities can be found in \cite[Lemma 6.1]{CNLMCDelCyl}.

    \proofstep{Step 1. The function $\Phi_1$ is well defined.} Notice first from the expression in \eqref{Phi1wG3} that the integrand is integrable in $\mathbb{R}$, and hence $\Phi_1(s)\in L^\infty(\mathbb{R})$. Indeed, by \eqref{Xineq1}, \eqref{Xineq2}, and \eqref{H5::UpperBound} we have, since $\beta> \alpha$,
    \begin{equation}
    \int_{\mathbb{R}}\abs{(\delta_-\varphi + \delta_+\varphi) G_3(a,t,\delta_-\varphi,\delta_+\varphi) \dd{t}} \leq 2\Lambda \norm{\varphi}_X \int_{\mathbb{R}}\min\{\abs{t},\abs{t}^{1+\beta}\}\abs{t}^{-2-\alpha} \dd{t} < + \infty.
    \end{equation}

    Now, to show that $\Phi_1(a,\varphi)$ is $(\beta- \alpha)$-Hölder continuous for every $\varphi \in X$ and every $a\in \mathbb{R}$, notice first that
    \begin{equation}\label{phi1_maps_onto_Y}
        2\{\Phi_1(a,\varphi)(s)-\Phi_1(a,\varphi)(\bar{s})\} = \int_{\mathbb{R}}\{i_1(s,\bar{s},t)+i_2(s,\bar{s},t)\}\dd{t},
    \end{equation}
    where
    \begin{equation}
        i_1(s,\bar{s},t)=\{(\delta_-\varphi+\delta_+\varphi)(s,t)-(\delta_-\varphi+\delta_+\varphi)(\bar{s},t)\}G_3(a,t,\delta_-\varphi(s,t),\delta_+\varphi(s,t))
    \end{equation}
    and
    \begin{equation}
        i_2(s,\bar{s},t)=(\delta_-\varphi+\delta_+\varphi)(\bar{s},t)\{G_3(a,t,\delta_-\varphi(s,t),\delta_+\varphi(s,t))-G_3(a,t,\delta_-\varphi(\bar{s},t),\delta_+\varphi(\bar{s},t))\}.
    \end{equation}

    We claim that there exists a constant $C>0$, depending only on $\abs{a}$, $\norm{\varphi}_X$, and $\Lambda$, such that
    \begin{equation}\label{claim_bound_i1_i2}
    \abs{i_1(s,\bar{s},t)}+\abs{i_2(s,\bar{s},t)}\leq C \min\{\abs{s-\bar{s}}\abs{t}^{\beta-2-\alpha},\abs{t}^{\beta-1-\alpha}\}.
    \end{equation}
    For $i_1$ the claim follows easily from \eqref{Xineq2}, \eqref{Xineq3}, and \eqref{H5::UpperBound}. For $i_2$, by \eqref{Xineq2} and \eqref{H5::UpperBound}, we have, on the one hand,
    \begin{equation}\label{first_bound_i2}
        \abs{i_2(s,\bar{s},t)}\leq C \abs{t}^{\beta-1-\alpha}.
    \end{equation}
    On the other hand, denoting $\Psi_t(s)=G_3(a,t,\delta_-\varphi(s,t),\delta_+\varphi(s,t))$, we have, by \eqref{Xineq2},
    \begin{equation}\label{bound_i2_MVTHM}
        \abs{i_2(s,\bar{s},t)}\leq 2\norm{\varphi}_X \abs{t}^{1+\beta} \abs{\Psi_t(s)-\Psi_t(\bar{s})}.
    \end{equation}
    Then, using \eqref{H6::DecayLipschitzNorm} and denoting $\xi_t(s,\tau)=-\delta_+\varphi(s,t)+(\delta_-\varphi(s,t)+\delta_+\varphi(s,t))\tau$, we get
    \begin{align}
    \abs{\Psi_t(s)-\Psi_t(\bar{s})} &\leq \int_0^1\abs{ K(t,a\xi_t(s,\tau))-K(t,a\xi_t(\bar{s},\tau))}\dd{\tau}\\  
    &\leq \frac{C}{|t|^{3+\alpha}} \abs{a} \int_0^1 \abs{\xi_t(s,\tau)-\xi_t(\bar{s},\tau)}\dd{\tau}\\ &\leq C \abs{s-\bar{s}}\abs{t}^{-3-\alpha},
    \end{align}
    where we have used the mean value theorem for $K(t,\cdot )$ and $\xi_t(\cdot,\tau)$, and $C>0$ is a constant depending only on $\abs{a}$, $\norm{\varphi}_X$, and $\Lambda$. Combining this with \eqref{bound_i2_MVTHM} and \eqref{first_bound_i2} we obtain \eqref{claim_bound_i1_i2}.

    Finally, we can show that $\Phi_1(a,\varphi)\in Y$ for every $\varphi \in X$. Indeed, we have already shown that $\Phi_1(a,\varphi)\in L^\infty(\mathbb{R})$. Moreover, from \eqref{phi1_maps_onto_Y} and \eqref{claim_bound_i1_i2}, and since $\beta>\alpha$, we have
    \begin{align}
        \abs{\Phi_1(a,\varphi)(s)-\Phi_1(a,\varphi)(\bar{s})} &\leq \frac{1}{2} \int_{\mathbb{R}}\left(\abs{i_1(s,\bar{s},t)}+\abs{i_2(s,\bar{s},t)}\right)\dd{t}\label{Phi1_is_HC_1}\\ 
        &\leq C \int_{\abs{t}\leq\abs{s-\bar{s}}} \abs{t}^{\beta-1-\alpha} \dd{t} + C \int_{\abs{t}\geq \abs{s-\bar{s}}}\abs{s-\bar{s}}\abs{t}^{\beta-2-\alpha} \dd{t} \\ 
        &= C\abs{s-\bar{s}}^{\beta-\alpha}\label{Phi1_is_HC_3},
    \end{align}
    for some constant $C>0$ depending only on $\abs{a}$, $\norm{\varphi}_X$, and $\Lambda$. Thus $\Phi_1(a,\varphi)$ is $(\beta-\alpha)$-Hölder continuous, whence $\Phi_1(a,\varphi)\in Y$.
    
    \proofstep{Step 2. Differentiability with respect to $a$.} We start by differentiating $\Phi_1$ with respect to $a$. From \eqref{Phi1wG3} we have

    \begin{equation} \label{DaPhi1}
        D_a \Phi_1(a,\varphi) = \frac{1}{2} \int_\mathbb{R} (\delta_-\varphi + \delta_+ \varphi) \partial_a G_3(a,t,\delta_-\varphi, \delta_+ \varphi) \dd{t}.
    \end{equation}
    Notice that the expression above is analogous to that of $\Phi_1$, but with $\partial_a G_3$ instead of $G_3$. With this in mind and denoting $\xi_t(s,\tau) = -\delta_+\varphi(s,t)+(\delta_-\varphi(s,t)+\delta_+\varphi(s,t))\tau$ as before, since 
    \begin{equation}
        \partial_a G_3(a,t,\delta_-\varphi,\delta_+\varphi) = \int_0^1 \xi_t(s,\tau) K_{z_2} (t,a\xi(s,\tau))\dd{\tau},
    \end{equation}
    by \eqref{Xineq1} and \eqref{H6::DecayLipschitzNorm} we have
    \begin{equation} \label{DaG3Decay}
        |\partial_a G_3(a,t,\delta_-\varphi,\delta_+\varphi)|\le C|t|^{-2-\alpha},
    \end{equation}
    for some $C>0$ depending only on $\|\varphi\|_X$ and $\Lambda$.

    Now, let $\tilde \Psi_t(s) = \partial_a G_3(a,t,\delta_-\varphi(s,t),\delta_+\varphi(s,t))$. By \eqref{H6::DecayLipschitzNorm} and \eqref{Xineq1} we have
    \begin{align}
        |\tilde \Psi_t(s)-\tilde \Psi_t(\bar s)| &\le \int_0^1 \left| \xi_t(s,\tau) K_{z_2}(t, a  \xi_t(s,\tau)) -   \xi_t(\bar s,\tau) K_{z_2}(t, a  \xi_t(\bar s,\tau))\right|\dd{\tau} \label{bound_tildePsi_1}\\
        &\le \begin{multlined}[t][0.7\displaywidth]
            \int_0^1 \left| \xi_t(s,\tau)\right| \left|K_{z_2}(t, a  \xi_t(s,\tau)) - K_{z_2}(t, a  \xi_t(\bar s,\tau))\right|\dd{\tau}\\
         + \int_0^1 \left|K_{z_2}(t, a  \xi_t(\bar s,\tau))\right| \left| \xi_t(s,\tau) -  \xi_t(\bar s,\tau)\right|\dd{\tau}
        \end{multlined}\\
        &\le \frac{C}{|t|^{4+\alpha}}\int_0^1 \left| \xi_t(s,\tau)\right| \left| a  \xi_t(s,\tau) - a  \xi_t(\bar s,\tau)\right|\dd{\tau} + \frac{C}{|t|^{3+\alpha}} \int_0^1 \left| \xi_t(s,\tau) -  \xi_t(\bar s,\tau)\right|\dd{\tau}\\
        &\le \frac{C}{|t|^{3+\alpha}}\int_0^1 \left|  \xi_t(s,\tau) -  \xi_t(\bar s,\tau)\right|\dd{\tau}\\
        &\le C  |s-\bar s| |t|^{-3-\alpha},\label{bound_tildePsi_2}
    \end{align}
    where we have used the mean value theorem for $\xi_t(\cdot,\tau)$, and $C$ denotes constants depending only on $\|\varphi\|_X$, $\abs{a}$, and $\Lambda$. This bound and \eqref{DaG3Decay} allow us to repeat the argument we used for $\Phi_1$ to prove that $D_a \Phi_1(a,\varphi)$ is well defined and belongs to $Y$.

    We now prove that $D_a \Phi_1$ is continuous from $\mathbb{R} \times X$ to $Y$. Let $\{(a_k, \varphi_k) \}_{k \in \mathbb{N}}$ be a sequence converging in $\mathbb{R} \times X$ to $(a,\varphi)$. We need to show that $D_a \Phi_1(a_k,\varphi_k)\to D_a \Phi_1(a,\varphi)$ in $Y$. First we bound the Hölder seminorm of the difference $D_a \Phi_1(a_k,\varphi_k)- D_a \Phi_1(a,\varphi)$, hence we must look at
    \begin{equation}
        (D_a\Phi_1(a_k,\varphi_k))(s) - (D_a\Phi_1 (a,\varphi))(s)-(D_a\Phi_1(a_k,\varphi_k))(\bar s)+D_a(\Phi_1 (a,\varphi))(\bar s).
    \end{equation}
    For every $t\in\mathbb{R}$, to simplify the notation, we write the integrand in \eqref{DaPhi1} as $T\phi(s) J\phi(s)$, with $\phi=(a,\varphi)$, $\phi_k=(a_k,\varphi_k)$, and
    \begin{align}
        T\phi(s) &:= (\delta_-\varphi + \delta_+\varphi)(s,t),\\
        J\phi(s) &:= \partial_a G_3(a,t,\delta_-\varphi, \delta_+ \varphi)(s,t).
    \end{align}
    Thus, we must bound the integral of
    \begin{equation}
      T\phi(s) J\phi(s) - T\phi_k(s) J\phi_k(s) - T\phi(\bar s) J\phi(\bar s) + T\phi_k(\bar s) J\phi_k(\bar s)  
    \end{equation}
   in $t$ over $\mathbb{R}$. To do this, we rewrite the expression above as
   \begin{align}
       T\phi&(s) (J\phi - J\phi_k)(s)\\
       &+ (T\phi-T\phi_k)(s)J\phi_k(s)\\
       &- T\phi(\bar s) (J\phi - J\phi_k)(\bar s)\\
       &- (T\phi-T\phi_k)(\bar s)J\phi_k(\bar s).
   \end{align}
   After adding and subtracting terms we get
   \begin{align} \label{GJGJ}
       (T&\phi(s) - T\phi(\bar s)) (J\phi - J\phi_k)(s)\\
       &+ (T\phi-T\phi_k)(s)(J\phi_k(s) - J\phi_k(\bar s))\\
       &+ T\phi(\bar s) ((J\phi - J\phi_k)(s) - (J\phi - J\phi_k)(\bar s))\\
       &+ ((T\phi-T\phi_k)(s) - (T\phi-T\phi_k)(\bar s))J\phi_k(\bar s).
   \end{align}
    A similar argument to the one leading to \eqref{claim_bound_i1_i2} (with $\partial_a G_3$ instead of $G_3$) shows that we can bound each line of the previous expression in absolute value by
    \begin{equation}
        c(k)\min\{\abs{s-\bar{s}}\abs{t}^{\beta-2-\alpha},\abs{t}^{\beta-1-\alpha}\},
    \end{equation}
    with $c(k)$ a sequence of constants such that $c(k) \to 0$ as $k \to +\infty$. Then, integrating in $t$ over $\mathbb{R}$ and proceeding as in \eqref{Phi1_is_HC_1}---\eqref{Phi1_is_HC_3} yields
    \begin{equation}
        \left[D_a\Phi_1(\phi)- D_a\Phi_1(\phi_k)\right]_{C^{\beta-\alpha}(\mathbb{R})} \to 0
    \end{equation}
    as $k\to \infty$.

    It remains to show that $D_a\Phi_1(\phi_k)\to D_a\Phi_1(\phi)$ uniformly as $k\to +\infty$. Notice first that
    \begin{equation}
        \abs{D_a\Phi_1(\phi)- D_a\Phi_1(\phi_k)} \leq \int_\mathbb{R} \abs{T\phi(s) (J\phi - J\phi_k)(s)+ (T\phi-T\phi_k)(s)J\phi_k(s)} \dd{t}.
    \end{equation}
    On the one hand, since the operator $T$ is linear, by \eqref{Xineq1}, \eqref{Xineq2}, and \eqref{DaG3Decay} we have
    \begin{equation}
        \int_{\mathbb{R}}\abs{(T\phi-T\phi_k)(s)J\phi_k(s)} \dd{t} \leq C \norm{\varphi-\varphi_k}_X \int_{\mathbb{R}}\min\{\abs{t},\abs{t}^{1+\beta}\}\abs{t}^{-2-\alpha}\dd{t} \to 0
    \end{equation}
    uniformly in $s$ as $k\to +\infty$. On the other hand, by a similar computation to \eqref{bound_tildePsi_1}---\eqref{bound_tildePsi_2} we have, denoting $\xi_t^k(s,\tau)=-\delta_+\varphi_k(s,t)+(\delta_-\varphi_k(s,t)+\delta_+\varphi_k(s,t))\tau$, and using the mean value theorem, \eqref{H6::DecayLipschitzNorm}, \eqref{Xineq1} and \eqref{Xineq2},
    \begin{align}
        \abs{T\phi(s) (J\phi - J\phi_k)(s)}&\le C \norm{\varphi}_X\abs{t} \int_0^1 \abs{\xi_t(s, \tau) K_{z_2}(t, a \xi_t(s, \tau)) - \xi_t^k(s, \tau) K_{z_2}(t, a \xi_t^k(s, \tau))}\dd{\tau}\\
        & \le \begin{multlined}[t]
            C\norm{\varphi}_X\abs{t} \left\{\frac{1}{|t|^{4+\alpha}}\int_0^1 \left| \xi_t(s,\tau)\right| \left| a  \xi_t(s,\tau) - a_k  \xi_t^k( s,\tau)\right|\dd{\tau} \right.\\ \left.+ \frac{1}{|t|^{3+\alpha}} \int_0^1 \left| \xi_t(s,\tau) -  \xi_t^k( s,\tau)\right|\dd{\tau}\right \}
        \end{multlined}\\ 
        &\leq \frac{C}{|t|^{2+\alpha}}\int_0^1 \left| \xi_t(s,\tau)\right| \left| a  - a_k  \right|\dd{\tau} + \frac{C}{|t|^{2+\alpha}}\int_0^1 \left| \xi_t(s,\tau) - \xi_t^k( s,\tau)\right|\dd{\tau}\\ 
        &\leq C\left\{\abs{a-a_k}+\norm{\varphi-\varphi_k}_X\right\}\min\{\abs{t},\abs{t}^{1+\beta}\}\abs{t}^{-2-\alpha},
    \end{align}
    and thus
    \begin{equation}
        \int_{\mathbb{R}}\abs{T\phi(s) (J\phi - J\phi_k)(s)} \dd{t} \leq C\left\{\abs{a-a_k}+\norm{\varphi-\varphi_k}_X\right\} \to 0
    \end{equation}
    uniformly in $s$ as $k\to +\infty$.

    This shows that $D_a\Phi_1(a_k,\varphi_k)$ converges to $D_a\Phi_1(a,\varphi)$ in $Y$ as $k\to +\infty$, and hence that $\Phi_1$ is continuously differentiable in $a$.

    \proofstep{Step 3. Differentiability with respect to $\varphi$.}
    We now prove the $C^1$ character of $\Phi_1$ with respect to $\varphi$. Denoting $\xi_t(s,\tau)=-\delta_+\varphi(s,t)+(\delta_-\varphi(s,t)+\delta_+\varphi(s,t))\tau$ again, from \eqref{Phi1wG3} we have
    \begin{multline}
        2 D_\varphi \Phi_1(a,\varphi)\psi = \int_\mathbb{R} (\delta_-\psi + \delta_+\psi)G_3(a,t,\delta_-\varphi,\delta_+\varphi)\dd{t}\\
         +\int_\mathbb{R} (\delta_-\varphi + \delta_+\varphi) \int_0^1 K_{z_2}(t,a\xi_t(s,\tau))a(-\delta_+\psi+(\delta_-\psi+\delta_+\psi)\tau)d \tau \dd{t}\label{definition_DphiPHI1}.
    \end{multline}

    Proceeding as in the beginning of Step 2, it is easily seen that
    \begin{equation}
        \norm{D_\varphi \Phi_1(a,\varphi)\psi}_{L^\infty(\mathbb{R})}\leq C \norm{\psi}_Y,
    \end{equation}
    and, using a similar argument as the one we did for $\Phi_1(a,\varphi)$, we can show that
    \begin{equation}
      \abs{(D_\varphi\Phi_1(a,\varphi)\psi)(s)-(D_\varphi\Phi_1(a,\varphi)\psi)(\bar s)}\leq C \norm{\psi}_Y\abs{s-\bar{s}}^{\beta-\alpha}.
    \end{equation}
    Hence, $D_\varphi \Phi_1(a,\varphi)$ defines a bounded linear operator between $X$ and $Y$.

    Lastly, we prove the continuity of $D_\varphi \Phi_1$. For this we must show that
    \begin{equation}
      \frac{\norm{D_\varphi \Phi_1(a,\varphi)\psi - D_\varphi \Phi_1(a_k,\varphi_k)\psi}_Y}{\norm{\psi}_X} \to 0
    \end{equation}
    as $k\to+\infty$. Hence, we must obtain bounds for
    \begin{equation}
      (D_\varphi \Phi_1(a_k,\varphi_k)\psi)(s) - (D_\varphi \Phi_1(a,\varphi)\psi)(s) - (D_\varphi \Phi_1(a_k,\varphi_k)\psi)(\bar s) + (D_\varphi \Phi_1(a,\varphi)\psi)(\bar s)
    \end{equation}
    and
    \begin{equation}
        (D_\varphi \Phi_1(a_k,\varphi_k)\psi)(s) - (D_\varphi \Phi_1(a,\varphi)\psi)(s)
    \end{equation}
    which are linear in $\norm{\psi}_X$. This can be achieved similarly to what we did for $\Phi_1(a,\varphi)(s)$ and $D_a\Phi_1(a,\varphi)(s)$ above, starting from \eqref{definition_DphiPHI1} and adding and subtracting terms appropriately.
\end{proof}

Lastly, we prove the $C^1$ character of $\Phi_2$. Recall from \eqref{Phi2} that
\begin{equation}\label{Phi2_bis}
    \Phi_2(a,\gamma,\varphi) = \int_\mathbb{R} \delta_0 \varphi G_2(t,a,\gamma, \delta_0 \varphi) \dd{t},
\end{equation}
with
\begin{equation}
    G_2(t,a,\gamma,r) = \int_0^1 K(t, 2\gamma R + a r \bar\tau) \dd{\bar{\tau}}.
\end{equation}
\begin{lemma}
    Assume that $K:\mathbb{R}^2 \to (0,+\infty)$ satisfies \eqref{H1::Symetry}---\eqref{H6::DecayLipschitzNorm}. There exists some $\tilde{\nu}>0$ small enough, depending only on $K$, for which the operator 
    \begin{equation}\label{G2_bis}
        \Phi_2 : (-\nu,\tilde{\nu})\times(1/2,3/2)\times B_{10}(0) \subset \mathbb{R} \times \mathbb{R} \times X \to Y    
    \end{equation}
     is of class $C^1$.
\end{lemma}

\begin{proof}
    Recall first that $R>0$ has been chosen depending only on $K$. Now, take $\tilde{\nu}>0$ such that $\tilde{\nu} \le R/40$ and denote 
    \begin{equation}
    \mathcal{U}=(-\tilde{\nu},\tilde{\nu})\times(1/2,3/2)\times B_{10}(0).
    \end{equation}
    Let
    $(a,\gamma,\varphi)\in \mathcal{U}$. Since $|\delta_0\varphi|\le 2\|\varphi\|_{L^\infty(0,\pi)} \le 20$ and $1/2<\gamma< 3/2$, we have, for every $\bar\tau \in [0,1]$,
    \begin{equation}
        \frac{R}{2}\le R - 20a\le 2\gamma R + a\delta_0\varphi\bar\tau \le 3R + 20a\le 4R.
    \end{equation}
    Hence, $K(t,2\gamma R+a \delta_0\varphi \bar\tau)$ is a positive, bounded function of $(t,a,\gamma,\delta_0\varphi)$ in $\mathbb{R}\times \mathcal{U}$, and therefore so is $G_2$. We denote by $M$ an upper bound for $K$ in this domain.

    Moreover, by \eqref{H5::UpperBound}, for every $(a,\gamma,\delta_0\varphi)\in \mathcal{U}$,
    \begin{equation}
      \abs{G_2(t,a,\gamma,\delta_0\varphi)}\leq \Lambda \abs{t}^{-2-\alpha}
    \end{equation}
    and therefore, the integrand in \eqref{Phi2_bis} is integrable and $\Phi_2$ is finite in $\mathcal{U}$.

    Next we check that $\Phi_2$ is well defined, i.e., that  $\Phi_2(a,\gamma,\varphi)\in Y$. In fact, we will see that $\Phi_2(a,\gamma,\varphi)$ is Lipschitz (and thus belongs to $Y$) if $\varphi$ is bounded and Lipschitz (which is always the case for $\varphi \in X$). We have already shown that $\Phi_2(a,\gamma,\varphi)\in L^\infty$. It remains to control the $C^{0,1}$ seminorm. For this, we write 
    \begin{multline}
     \Phi_2(a,\gamma,\varphi)(s)-\Phi_2(a,\gamma,\varphi)(\bar s)=   \int_\mathbb{R} (\delta_0 \varphi(s,t)- \delta_0 \varphi(\bar s,t))G_2(t,a,\gamma,\delta_0 \varphi(s,t))\dd{t}\\
        - \int_\mathbb{R} \delta_0 \varphi(\bar s,t) (G_2(t,a,\gamma,\delta_0 \varphi(\bar s,t))-G_2(t,a,\gamma,\delta_0 \varphi(s,t)))\dd{t}.\label{Phi2_lip}
    \end{multline}
    The first term can be controlled using that $G_2$ is integrable and
    \begin{equation}\label{delta_0phi_lip}
      |\delta_0\varphi(s,t)-\delta_0\varphi(\bar s,t)| \le 2\norm{\varphi}_{C^{0,1}(\mathbb{R})}|s-\bar s|.  
    \end{equation}
    For the second one, we use \eqref{H6::DecayLipschitzNorm} and \eqref{delta_0phi_lip} to get
    \begin{align}
        \abs{G_2(t,a,\gamma,\delta_0 \varphi(s,t))-G_2(t,a,\gamma,\delta_0 \varphi(\bar s,t))} \leq \\ 
        &\hspace{-3 cm}\leq \int_0^1 \abs{K(t, 2\gamma R + a \delta_0\varphi ( s, t)) \bar\tau)-K(t, 2\gamma R + a \delta_0\varphi(\bar s,t) \bar\tau)} \dd{\bar{\tau}}\\ 
        &\hspace{-3 cm}\leq \min\{\Lambda \abs{t}^{-3-\alpha},M\} \int_0^1 \abs{a}\abs{\delta_0\varphi(s,t)-\delta_0\varphi(\bar s,t)}\bar \tau d \bar\tau\\ 
        &\hspace{-3 cm}\leq C \min\{1,\abs{t}^{-3-\alpha}\} \abs{s-\bar s}
    \end{align}
    for some $C>0$ depending only on $\abs{a}$, $\norm{\varphi}_{C^{0,1}(\mathbb{R})}$, $\Lambda$, and $R$. Hence we deduce from \eqref{Phi2_lip} that
    \begin{equation}
        \abs{\Phi_2(a,\gamma,\varphi)(s)-\Phi_2(a,\gamma,\varphi)(\bar s)}\leq C \abs{s-\bar s},
    \end{equation}
    and, as a consequence, $\Phi_2(a,\gamma,\varphi)\in Y$.

    Lastly, we prove that $\Phi_2$ is $C^1$ with respect to $(a,\gamma,\varphi)$ with values in $Y$. We start by computing $D_a \Phi_2$, $D_\gamma \Phi_2$ and $D_\varphi \Phi_2$ at $(a,\gamma,\varphi)$ as we did in the proof of Lemma \ref{Phi1C1}. An analogous argument to the one for $\Phi_2 (a,\gamma,\varphi)$ yields that these functions are Lipschitz and have Lipschitz norm controlled by $|a|+\gamma+\|\varphi\|_{C^{0,1}(\mathbb{R})}$. This allows us to prove the continuity of $D_a \Phi_2,D_\gamma \Phi_2$ and $D_\varphi\Phi_2$ as functions of $(a,\gamma,\varphi)$ in $Y$ as follows: If $(a_k,\gamma_k,\varphi_k)\to (a,\gamma,\varphi)$ in $\mathbb{R}\times\mathbb{R}\times X$, then, by the previous assertion the sequence $D_{(a,\gamma,\varphi)}\Phi_2(a_k,\gamma_k,\varphi_k)$ is uniformly bounded in the Lipschitz norm, and hence, by the Arzelà-Ascoli theorem, there exists a subsequence converging in the weaker Hölder norm of $Y$ to some function in $Y$. On the other hand, applying the dominated convergence theorem, we deduce that this function in $Y$ is necessarily $D_{(a,\gamma,\varphi)}\Phi_2(a,\gamma,\varphi)$, as the integrals appearing in the expression of $D_{(a,\gamma,\varphi)}\Phi_2(a_k,\mu_k,\varphi_k)$ converge to the ones in the expression of $D_{(a,\gamma,\varphi)}\Phi_2(a,\gamma,\varphi)$. Therefore, the full sequence $D_{(a,\gamma,\varphi)}\Phi_2(a_k,\gamma_k,\varphi_k)$ converges to $D_{(a,\gamma,\varphi)}\Phi_2(a,\gamma,\varphi)$ in $Y$.
\end{proof}
\section*{Acknowledgements}
The authors thank Xavier Cabré for his guidance and useful discussions on the topic of this paper.

\end{document}